\documentclass[reqno, 10pt]{amsart}

\usepackage{a4wide}
\usepackage{hyperref}

\numberwithin{equation}{section}

\theoremstyle{plain}
\newtheorem{theorem}{Theorem}[section]
\newtheorem{corollary}[theorem]{Corollary}
\newtheorem{proposition}[theorem]{Proposition}
\newtheorem{lemma}[theorem]{Lemma}

\theoremstyle{definition}
\newtheorem{definition}[theorem]{Definition}

\theoremstyle{remark}
\newtheorem{remark}[theorem]{Remark}

\newcommand\bep{\ensuremath{\boldsymbol{\epsilon}}}
\newcommand\bepbar{\ensuremath{\overline{\boldsymbol{\epsilon}}}}
\newcommand\bth{\ensuremath{\boldsymbol{\theta}}}

\newcommand\bpi{\ensuremath{\boldsymbol{\pi}}}
\newcommand\bro{\ensuremath{\boldsymbol{\rho}}}
\newcommand\bmu{\ensuremath{\boldsymbol{\mu}}}
\newcommand\bnu{\ensuremath{\boldsymbol{\nu}}}
\newcommand\bOm{\ensuremath{\boldsymbol{\Omega}}}
\newcommand\bom{\ensuremath{\boldsymbol{\omega}}}

\newcommand\bvar{\ensuremath{\boldsymbol{\varphi}}}

\newcommand{\llambda}[3]{\ensuremath{\lambda^{#2}{}_{{#3}{#1}}}}
\newcommand{\llambdatilde}[3]{\ensuremath{\widetilde{\lambda}^{#2}{}_{{#3}{#1}}}}
\newcommand{\GGamma}[3]{\ensuremath{\Gamma^{#2}{}_{{#3}{#1}}}}
\newcommand{\boldpi}[3]{\ensuremath{{\boldsymbol{\pi}}^{#2}{}_{{#3}{#1}}}}
\newcommand{\boldpitilde}[3]{\ensuremath{\widetilde{\boldsymbol{\pi}}^{#2}{}_{{#3}{#1}}}}

\newcommand\restrict[2]{\ensuremath{\left. {#1} \right| \! {#2} }}

\newcommand\codim{\ensuremath{\mathrm{codim}\,}}
\newcommand\rank{\ensuremath{\mathrm{rank}\,}}
\newcommand\sign{\ensuremath{\mathrm{sign}}}
\newcommand\spn{\ensuremath{\mathrm{span}\,}}

\newcommand\XDS{exterior differential system}
\newcommand\XDSIC{exterior differential system with independence condition}

\newcommand\FX{\ensuremath{\mathcal{F}}}

\newcommand\vs{\vskip .2cm}

\begin{document}
\title{Block diagonalisation of four-dimensional metrics}
\author[J.D.E.~Grant]{James~D.E.\ Grant}
\address{James~D.E.\ Grant: \href{http://www.mat.univie.ac.at/home.php}{Fakult{\"a}t f{\"u}r Mathematik} \\
\href{http://www.univie.ac.at/de/}{Universit{\"a}t Wien} \\ Nordbergstrasse 15 \\ 1090 Wien \\ Austria}
\email{\href{mailto:james.grant@univie.ac.at}{james.grant@univie.ac.at}}
\author{J.A.\ Vickers}
\address{J.A.\ Vickers: \href{http://www.soton.ac.uk/maths/index.shtml}{School of Mathematics}, \href{http://www.soton.ac.uk/}{University of Southampton} \\ Southampton SO17 1BJ \\ UK}
\email{\href{mailto:jav@maths.soton.ac.uk}{jav@maths.soton.ac.uk}}
\thanks{This work was partially supported by START-project Y237--N13 of the \href{http://www.fwf.ac.at/}{Austrian Science Fund}.}

\begin{abstract}
It is shown that, in $4$-dimensions, it is possible to introduce
coordinates so that an analytic metric locally takes block diagonal
form. i.e. one can find coordinates such that $g_{\alpha\beta} = 0$ for
$(\alpha, \beta) \in S$ where $S = \left\{ (1, 3), (1, 4), (2, 3), (2, 4)
\right\}$. We call a coordinate system in which the metric takes this form
a \lq doubly biorthogonal coordinate system\rq. We show that all such
coordinate systems are determined by a pair of coupled second-order partial
differential equations.
\end{abstract}
\maketitle
\thispagestyle{empty}

\section{Introduction}
\label{sec:intro}

This paper is concerned with making coordinate choices to put general
metrics into simplified or canonical forms. A metric in $2$-dimensions
depends upon $\tfrac{1}{2} \times 2 (2 + 1) = 3$ arbitrary functions
$g_{11}$, $g_{12}$ and $g_{22}$. On the other hand, the diffeomorphism
freedom
\begin{align*}
f \colon \mathbb{R}^2 &\rightarrow \mathbb{R}^2
\\
\left( x, y \right) &\mapsto \left( f_1(x, y), f_2(x, y) \right)
\end{align*}
contains $2$ arbitrary functions. Given any $2$-dimensional metric, one
would therefore expect to be able to introduce local coordinates such that
the metric depended on only $3 - 2 = 1$ function. Indeed, it is a classical
result that, in two dimensions, every metric is (locally) conformally flat,
i.e. there exist coordinates so that
\[
ds^2 = \Omega^2(x, y) \left( dx^2 + dy^2 \right).
\]
The proof for analytic metrics goes back to Gauss~\cite{Gauss}, while the
proof for smooth metrics is more recent (see for example~\cite{Spivak} for
details).

In $3$-dimensions the metric depends upon $\tfrac{1}{2} \times 3 (3 + 1) =
6$ arbitrary functions, while the diffeomorphism freedom $f \colon \mathbb{R}^3
\rightarrow \mathbb{R}^3$ involves $3$ functions. One would therefore
expect to be able to introduce coordinates such that a $3$-dimensional
metric was specified by $6 - 3 = 3$ functions. In fact in $3$-dimensions
one can introduce coordinates that locally diagonalise the
metric. i.e. there exist coordinates such that
\[
ds^2 = A(x, y, z) dx^2 + B(x, y, z) dy^2 + C(x, y, z) dz^2.
\]
Again the proof of this result in the analytic case goes back a long
way~\cite{Cartan}. The proof in the smooth case was, again, much more
recent~\cite{DeTY} and uses the theory of the characteristic variety of an
exterior differential system.

In $4$-dimensions the metric depends upon $\tfrac{1}{2} \times 4(4+1) = 10$
arbitrary functions, while the diffeomorphism freedom $f \colon \mathbb{R}^4
\rightarrow \mathbb{R}^4$ gives $4$ functions. One would therefore expect
to be able to write a $4$-dimensional metric in a canonical form that
depended upon $10 - 4 = 6$ arbitrary functions. Thus, in general, one
cannot expect to be able to diagonalise a metric in $4$-dimensions,
although of course in special cases this is possible (this problem was
considered in~\cite{Tod}). However it was suggested to one of us by David
Robinson that an appropriate local canonical form for $4$-dimensional
metrics was the \lq block diagonal\rq\ form
\[
g_{\alpha\beta} =
\left(\begin{array}{c c c c}
A & B & 0 & 0 \\
B & C & 0 & 0 \\
0 & 0 & D & E \\
0 & 0 & E & F
\end{array}\right).
\]
In this paper we use Cartan's theory of exterior differential systems to
show that it is indeed possible to write an analytic $4$-dimensional metric
in this form, at least locally. We show that the problem of finding local
coordinates that block-diagonalise a metric may be reformulated as a
condition on an orthonormal tetrad (see equation~(\ref{dep})). From this
reformulation, we construct an exterior differential system on the
orthonormal frame bundle of our manifold, the integral manifolds of which give rise
to solutions of our block-diagonalisation problem.  This exterior
differential system is not involutive, however, so we must go to the first
prolongation. At this point, we discover a consistency condition for our
system, (\ref{constraint1}), that must be satisfied. Imposing this
constraint on our exterior differential system gives rise to an involutive
Pfaffian system, to which the Cartan--K\"{a}hler theorem may be applied to
show existence of solutions. Note that the consistency condition mentioned
above may be interpreted on our manifold as a relation between a curvature
component and various components of the connection
(cf. equation~(\ref{curvcondition}) for the Riemannian version of this
constraint and equation~(\ref{NPconstraint}) for the Lorentzian version in
Newman--Penrose formalism).  At the level of our four-dimensional manifold,
this constraint may be deduced directly as being a consequence of the
conditions~(\ref{dep}) imposed on the orthonormal tetrad.  The constraint
involves the extrinsic curvature of the two surfaces and does not
impose any additional geometrical restrictions on our manifold. Indeed the
fact that we have a Pfaffian system on the first prolongation which
satisfies the conditions for the Cartan--K\"{a}hler theorem shows that
the block-diagonalisation of \emph{any\/} four-dimensional metric
may be carried out locally.

Although our results for local
canonical forms have assumed that the metric is Riemannian, they remain
true in the Lorentzian case (with obvious modifications). Similarly, we
will assume that our metric is Riemannian, although the proof may easily be
adapted for metrics of Lorentzian or $(-, -, +, +)$ signature. The
Lorentzian version of the $4$-dimensional result is, in particular, useful
in establishing certain results in general relativity. For example, it can
be used to establish some results concerning the geometry of generalised
cosmic strings~\cite{Kini} and can also be used to make a gauge choice
within the $2+2$ formalism~\cite{dIS} in which all the shifts
$\beta^i_\alpha$ vanish.

Given a local canonical form for a metric one can ask what transformations
preserve that form. For the case of a metric in $2$-dimensions a conformal
(in the sense of complex analytic) transformation of the flat metric will
map isothermal coordinates into isothermal coordinates. Similarly in the
$3$-dimensional case the problem is essentially the same as finding all \lq
triply orthogonal coordinate systems\rq\ which are coordinates in which the
flat metric is diagonal. The problem of finding all such coordinate systems
was solved by Darboux~\cite{Darboux}, who showed that it required the
solution of a certain third-order partial differential equation. Similarly
in $4$-dimensions the problem is essentially the same as finding all \lq
doubly biorthogonal coordinate systems\rq\ which are coordinates in which
the flat metric is in block diagonal form. We will show in
Section~\ref{sec:double} that all such coordinates are determined by the
solution of a pair of coupled second-order equations.

\vs
The plan of this paper is as follows. In Section~\ref{sec:3d} we briefly
review the proofs that $3$-dimensional metrics may be diagonalised in both
the analytic and smooth case. In Section~\ref{sec:block} we explain why
these methods fail to give a direct proof of the block diagonalisation of a
$4$-dimensional metric. However, we reduce the problem of block
diagonalising a metric to the problem of constructing an orthonormal tetrad
that satisfies a particular set of identities~\eqref{dep}. In
Section~\ref{sec:XDS} we show, using the theory of exterior differential
systems, that, in the case where the metric is analytic, such an
orthonormal tetrad can always be constructed. As such, we deduce that a
four-dimensional analytic metric can be block-diagonalised. In
Section~\ref{sec:double} we discuss triply orthogonal systems in
$3$-dimensions to motivate the discussion of doubly biorthogonal systems of
coordinates in $4$-dimensions. In order to make the paper reasonably
self-contained, we have collected together the main background material
that we require from the theory of exterior differential systems in Appendix~\ref{app:XDS}.

\vs
\noindent{\textbf{Notation}}: In the earlier sections of this paper, we
will often have cause to refer to a single diagonal component
$g_{\alpha\alpha}$ of a metric. Also, when working with exterior
differential systems, it is sometimes convenient to explicitly write out
the terms in a sum individually, rather than use the summation
convention. Therefore, we will generally \emph{\textbf{not}\/} use the
Einstein summation convention in this paper, with the exception
of Section~\ref{sec:double}, where the above issues do not arise.

\vs Note also that we will use Greek letters for coordinate indices and
Latin letters for frame indices.

\section{Diagonalising metrics in $3$-dimensions}
\label{sec:3d}

In this section, we review the methods of proving that a $3$-dimensional
smooth metric can be diagonalised.

In the analytic case, rather than working with the covariant metric
$g_{\alpha\beta}$ it is more convenient to consider the equivalent problem
of diagonalising the contravariant metric $g^{\alpha\beta}$. Given
$g^{\alpha\beta}(x^1, x^2, x^3)$ we wish to find new coordinates $\{
x^{\alpha'}(x^1, x^2, x^3) : \alpha = 1, 2, 3 \}$ such that
\[
g^{\alpha'\beta'} = \sum_{\gamma, \delta} \frac{\partial
x^{\alpha'}}{\partial x^\gamma} \frac{\partial x^{\beta'}}{\partial
x^\delta} g^{\gamma\delta} = 0 \quad \hbox{ for } \alpha' \neq \beta'.
\]
This is a non-linear system of $3$ equations (taking $(\alpha', \beta')$ to
be $(1, 2)$, $(1, 3)$ and $(2, 3)$) for three unknowns $x^{1'}$, $x^{2'}$
and $x^{3'}$. In the analytic case one can show that solutions to these
equations exist but the solutions are not unique (there are trivial
transformations given by replacing $x^{1'}$, $x^{2'}$ and $x^{3'}$ with
$h^1(x^{1'})$, $h^2(x^{2'})$ and $h^3(x^{3'})$) and the strongly non-linear
nature of the equations makes it hard to utilise this method in the smooth
case. Instead, DeTurck and Yang~\cite{DeTY} seek an orthonormal coframe
${\bep}^1$, ${\bep}^2$, ${\bep}^3$, and a coordinate system $x^1$, $x^2$,
$x^3$ such that
\begin{equation}
{\bep}^i = f^i dx^i, \qquad i = 1, 2, 3.
\label{1}
\end{equation}
(Recall \textbf{no summation}.) Clearly such a frame would imply that
$g_{\mu\nu}$ is diagonal in the coordinate system of the $x^\mu$.

The advantage of condition~\eqref{1} is that, by the Frobenius theorem, it
is (locally) equivalent to the existence of a coframe such that
\begin{equation}
{\bep}^i \wedge d{\bep}^i = 0, \qquad i = 1, 2, 3
\label{2}
\end{equation}
and this is a problem that may be solved without having to consider
coordinate transformations. Furthermore, one would expect the ${\bep}^i$ to
be unique (up to relabelling) since the lack of uniqueness in the
coordinates noted above is absorbed into the $f^i$.

Let $\{ {\bepbar}^i \}$ be some fixed orthonormal frame for
$g_{\alpha\beta}$ in some open set. Then, since ${\bep}^i$ and
${\bepbar}^i$ are both orthonormal, they are related by some
$\mathrm{SO}(3)$ transformation $a^i{}_j$
\begin{equation}
{\bep}^i(x) = \sum_j a^i{}_j(x) \, {\bepbar}^j(x).
\label{3}
\end{equation}
We now substitute~\eqref{3} into~\eqref{2} to obtain
\begin{equation}
\sum_{j, k} a^i{}_j \, {\bepbar}^j \wedge d \left( a^i{}_k \, {\bepbar}^k
\right) = 0, \qquad i = 1, 2, 3.
\label{4}
\end{equation}
Note that this gives $3$ equations for $3$ unknowns (such as the Euler
angles) which parameterise elements of $\mathrm{SO}(3)$. To show that there
exist solutions to~\eqref{4}, DeTurck and Yang write the second term as
\[
d \left(a^i{}_k \, {\bepbar}^k \right) =
\sum_l \left( a^i{}_{k|l} \, {\bepbar}^l \wedge {\bepbar}^k + a^i{}_k \, d{\bepbar}^k \right)
\]
(where $f_{|i} = \overline{\mathbf{e}}_i \left( f \right)= \sum_\mu
\overline{e}_i{}^\mu \tfrac{\partial f}{\partial x^{\mu}}$, with
$\overline{\mathbf{e}}_i$ the dual basis to ${\bepbar}^i$). They then use
Cartan's first structure equation to write
\[
d{\bepbar}^k = \sum_{l, m} \overline{\gamma}^k{}_{lm} \, {\bepbar}^l
\wedge {\bepbar}^m,
\]
where $\overline{\gamma}^k{}_{ml}$ are the connection coefficients with
respect to the frame ${\bepbar}^i$. (Our conventions are that $d{\bep}^i =
- \Gamma^i{}_j \wedge {\bep}^j$ with $\Gamma^i{}_j = \sum_k
\gamma^i{}_{jk} {\bep}^k$.)

Substituting in~\eqref{4} gives
\[
\sum_{\sigma \in \Sigma_3} \sum_{j, k, l, m} \left( \sign\, \sigma \right)
a^i{}_{\sigma(j)}
\left( a^i{}_{\sigma(k)|\sigma(l)} + a^i{}_m
\overline{\gamma}^m{}_{\sigma(k)\sigma(l)} \right) = 0,
\qquad i = 1, 2, 3.
\]
One can then solve for $a^i{}_{k|l}$ and show that the resulting system is
diagonal hyperbolic (a special case of symmetric hyperbolic). In the smooth
case one has existence and uniqueness theorems for such systems of
equations (see e.g.~\cite{TaylorPsiD}), so that one can show the existence
of a unique (up to relabelling) orthonormal frame satisfying~\eqref{1} and
hence a diagonal metric. Note however that as remarked earlier the coordinate
expression~\eqref{1} is not unique, but one is free to replace $x^1$ by
$h(x^1)$ etc, so the actual diagonal entries of the metric are not unique.

\section{Block diagonalisation of $4$-dimensional metrics}
\label{sec:block}

In this section and the next, we shall show that it is possible, in the
analytic case, to introduce coordinates that block-diagonalise a
$4$-dimensional metric. The proof will eventually be by an application of
the Cartan--K\"{a}hler theorem, a generalisation of the Cauchy--Kovalevskya
theorem~\cite{BCGGG}. However, we shall begin by trying to repeat the
methods for diagonalising analytic metrics in $3$-dimensions.

Given $g^{\alpha\beta}(x^1, x^2, x^3, x^4)$, we want to find new
coordinates $\{ x^{\alpha'}(x^1, x^2, x^3, x^4): \alpha=1, \dots, 4 \}$
such that
\[
g^{\alpha'\beta'} =
\sum_{\gamma, \delta} \frac{\partial x^{\alpha'}}{\partial x^\gamma}
\frac{\partial x^{\beta'}}{\partial x^\delta} g^{\gamma\delta} = 0
\quad \hbox{ for } (\alpha', \beta') \in S,
\]
where $S=\{(1, 3), (1, 4), (2, 3), (2, 4)\}$. This gives $4$ equations for
$4$ unknowns.

For ease of notation, we let $x^{\alpha'}(x^1, x^2, x^3, x^4) =
y^{\alpha}(x^1, x^2, x^3, x^4) = y^\alpha(x^\beta)$. We now linearise about
$y_o^\alpha(x^\beta)$ and obtain
\[
\sum_{\gamma, \delta} \left( y^\alpha_{, \gamma} y^\beta_{o, \delta}
+ y^\beta_{, \gamma} y^\alpha_{o, \delta} \right) g^{\gamma\delta}
= - \sum_{\gamma, \delta} y^\alpha_{o, \gamma} y^\beta_{o, \delta}
g^{\gamma\delta} \quad
\hbox{ for } (\alpha, \beta) \in S.
\]
This is a system of the form
\[
P^\alpha \frac{\partial}{\partial x^\alpha}{\mathbf{y}} = \mathbf{c},
\]
where
\begin{align*}
P^\alpha&=\left(\begin{array}{c c c c}
y_o^{3, \alpha} & 0 & y_o^{1, \alpha} & 0 \\
y_o^{4, \alpha} & 0 & 0 & y_o^{1, \alpha} \\
0 & y_o^{3, \alpha} & y_o^{2, \alpha} & 0 \\
0 & y_o^{4, \alpha} & 0 & y_o^{2, \alpha}
\end{array}\right),
\\
{}&{}\\
\mathbf{c}&=\left(\begin{array}{c c c c }
-y^1_{o, \alpha}y^3_{o, \beta}g^{\alpha\beta} \\
-y^1_{o, \alpha}y^4_{o, \beta}g^{\alpha\beta} \\
-y^2_{o, \alpha}y^3_{o, \beta}g^{\alpha\beta} \\
-y^2_{o, \alpha}y^4_{o, \beta}g^{\alpha\beta}
\end{array}\right),
\end{align*}
and $y_o^{\beta, \alpha}=y^\beta_{o, \gamma}g^{\alpha\gamma}$.

Unfortunately, when one attempts to find the characteristic surfaces, one finds
\[
\det( P^\alpha \xi_\alpha ) = 0, \qquad \forall \xi_\alpha \in \mathbb{R}^4,
\]
so that there are no non-characteristic surfaces and the initial data must
satisfy some constraint. As a result, one cannot directly apply the
Cauchy--Kovalevskya theorem, unlike in the apparently similar problem of
diagonalising a metric in $3$-dimensions.

\vs We therefore turn to the method of DeTurck and Yang. In this case, this
involves finding a coframe $\{ {\bep}^i \}$ and a coordinate system such
that
\begin{subequations}
\begin{align}
{\bep}^1 \wedge {\bep}^2 &= f dx^1 \wedge dx^2,
\label{3.1a}
\\
{\bep}^3 \wedge {\bep}^4 & = g dx^3 \wedge dx^4.
\label{3.1b}\end{align}\end{subequations}
Note that~\eqref{3.1a} implies
\begin{equation}
{\bep}^i = \sum_{\mu = 1, 2} {\bep}^i_\mu dx^\mu \qquad i = 1, 2
\label{3.2a}
\end{equation}
and that~\eqref{3.1b} implies
\begin{equation}
{\bep}^i = \sum_{\mu = 3, 4} {\bep}^i_\mu dx^\mu \qquad i = 3, 4
\label{3.2b}
\end{equation}
and hence $g_{\mu\nu} = \sum_{i, j} \delta_{ij} {\bep}^i_\mu {\bep}^j_\nu$
is block diagonal. Conversely, if $g_{\mu\nu}$ is block diagonal, we can
certainly find a coframe that satisfies~\eqref{3.2a} and~\eqref{3.2b} and
hence~\eqref{3.1a} and~\eqref{3.1b}.

\vs This leads to the following characterisation of metrics that can be
block-diagonalised:

\begin{proposition}
\label{prop:1}
A Riemannian metric $\mathbf{g}$ can be block-diagonalised if and only if
it admits an orthonormal coframe, $\{ {\bep}^a: a = 1, \dots, 4 \}$, that
satisfies the relations
\begin{equation}
\begin{aligned}
{\bep}^1 \wedge {\bep}^2 \wedge d{\bep}^1 &= 0,
\\
{\bep}^1 \wedge {\bep}^2 \wedge d{\bep}^2 &= 0,
\\
{\bep}^3 \wedge {\bep}^4 \wedge d{\bep}^3 &= 0,
\\
{\bep}^3 \wedge {\bep}^4 \wedge d{\bep}^4 &= 0.
\end{aligned}
\label{dep}
\end{equation}
\end{proposition}
\begin{proof}
Given a coframe that obeys relations~\eqref{dep}, the Frobenius theorem
implies the existence of local coordinates $(t, x, y, z)$ and functions
$\alpha, \dots, \theta$ such that
\begin{equation}
\begin{aligned}
{\bep}^1 =
\alpha \, dt + \beta \, dx, \qquad {\bep}^2 = \gamma \, dt + \delta \, dx,
\\
{\bep}^3 = \epsilon \, dy + \zeta \, dz, \qquad {\bep}^4
= \eta \, dy + \theta \, dz.
\end{aligned}
\label{coframe}
\end{equation}
The metric $\mathbf{g}$ is then block-diagonal in this coordinate
system. Conversely, if the metric $\mathbf{g}$ is block-diagonal with
respect to a coordinate system $(t, x, y, z)$, then we can choose a coframe
of the form~\eqref{coframe}, which then automatically
satisfies~\eqref{dep}.
\end{proof}

\begin{remark}
Although we have stated the block-diagonalisation problem in terms of Riemannian manifolds,
it is clear that the problem of block-diagonalising a metric is conformally invariant, In particular,
a coordinate system that block-diagonalises a representative metric in a conformal equivalence class
will block-diagonalise all representatives in that conformal equivalence class.
We will pursue the Riemannian version of the problem for simplicity,
although all of our calculations can be reformulated in a conformally equivariant fashion.
\label{rem:cfl}
\end{remark}

In the next section, the characterisation given in Proposition~\ref{prop:1} will be used to show that all
analytic four-dimensional metrics can be block-diagonalised.

\section{Exterior differential systems}
\label{sec:XDS}

In this section, we use the theory of exterior differential systems, in
particular the Cartan--K\"{a}hler theorem, to show that, for a given
analytic metric $\mathbf{g}$, we can find an orthonormal coframe that
satisfies the conditions~\eqref{dep} of Proposition~\ref{prop:1}. Our
notation, generally, follows that of~\cite{BCGGG}. The methods that we use
are similar to those used in the study of orthogonal coordinates for
Riemannian metrics in Chapter~III, Section~3, Example~3.2, and Chapter~VII,
Section~3 of~\cite{BCGGG}. For completeness, however, a summary of the
relevant terminology and results from exterior differential systems theory
has been included in Appendix~\ref{app:XDS}.

\vskip .2cm Let $X$ be an oriented four-manifold with a Riemannian metric
$\mathbf{g}$, and let $\pi \colon \FX \rightarrow X$ be the bundle of orthonormal
coframes of $(X, \mathbf{g})$. We will denote points in $\FX$ by either $p$
or, since we are working locally, we will assume a trivialisation
$\pi^{-1}(X) \cong X \times \mathrm{SO}(4)$ and denote points in $\FX$ by
$(x, g)$ where $x \in X$ and $g \in \mathrm{SO}(4)$. The bundle $\FX$ comes
equipped with a canonical basis of $1$-forms consisting of the components,
$\{ {\bom}^a \}_{a = 1, \dots, 4}$, of the tautological $1$-form on $\FX$
and the components, $\{ {\bom}^a{}_b \}_{a, b = 1, \dots, 4}$, of the
Levi-Civita connection (see, e.g.,~\cite{IL}). These differential forms
have the following properties:
\begin{itemize}
\item[$\bullet$] \textbf{Reproducing property}: An orthonormal coframe $\{
\bep^a \}_{a=1}^4$ on $M$ defines a corresponding section $f \colon X \rightarrow
\FX$. Pulling back the tautological $1$-forms on $\FX$ by this section
reproduces the coframe $\{ \bep^a \}$ i.e. $f^* {\bom}^a = {\bep}^a$.
\item[$\bullet$] \textbf{Canonical coframing}: A canonical coframing of
$\FX$ consists of the tautological $1$-forms ${\bom}^a$, $a = 1, \dots, 4$
and the connection $1$-forms ${\bom}^a{}_b$, where $a, b = 1, \dots, 4$
with $a < b$. Note that we will often write summations that involve terms
of the form ${\bom}^a{}_b$ with $a > b$. In this case, we identify
${\bom}^a{}_b$ with $- \sum_{c, d} \delta^{ac} \delta_{bd} {\bom}^d{}_c$,
consistent with the $\mathrm{SO}(4)$ nature of the connection. We adopt
similar conventions with quantities such as $\llambda{a}{b}{c}$ introduced
later.
\item[$\bullet$] \textbf{Cartan structure equations}: The one-forms $\{
\bom^a, {\bom}^a{}_b \}$ obey the Cartan structure equations
\begin{align*}
d{\bom}^a + \sum_b {\bom}^a{}_b \wedge {\bom}^b &= 0,
\\
d{\bom}^a{}_b + \sum_c {\bom}^a{}_c \wedge {\bom}^c{}_b &= \mathbf{\Omega}^a{}_b ,
\end{align*}
where
\[
\mathbf{\Omega}^a{}_b = \frac{1}{2} \sum_{c, d} R^a{}_{bcd} \, {\bom}^c \wedge {\bom}^d \in \Omega^2(\FX, \mathfrak{so}(4))
\]
is the curvature form of the connection form $\bom^a{}_b$.
(Recall our convention mentioned above for ${\bom}^a{}_b$ with $a > b$.)
\end{itemize}

\vs Following Proposition~\ref{prop:1}, let $\mathcal{I} \subset
\Omega^*(\FX)$ be the exterior differential system on $\FX$ generated by
the $4$-forms
\begin{equation}
\begin{aligned}
\boldsymbol{\Theta}^1 &:= {\bom}^1 \wedge {\bom}^2 \wedge d{\bom}^1 =
{\bom}^1 \wedge {\bom}^2 \wedge {\bom}^3 \wedge {\bom}^1{}_3 + {\bom}^1 \wedge {\bom}^2 \wedge {\bom}^4 \wedge {\bom}^1{}_4,
\\
\boldsymbol{\Theta}^2 &:= {\bom}^1 \wedge {\bom}^2 \wedge d{\bom}^2 =
{\bom}^1 \wedge {\bom}^2 \wedge {\bom}^3 \wedge {\bom}^2{}_3 + {\bom}^1 \wedge {\bom}^2 \wedge {\bom}^4 \wedge {\bom}^2{}_4,
\\
\boldsymbol{\Theta}^3 &:= {\bom}^3 \wedge {\bom}^4 \wedge d{\bom}^3 =
- {\bom}^1 \wedge {\bom}^3 \wedge {\bom}^4 \wedge {\bom}^1{}_3 - {\bom}^2 \wedge {\bom}^3 \wedge {\bom}^4 \wedge {\bom}^2{}_3,
\\
\boldsymbol{\Theta}^4 &:= {\bom}^3 \wedge {\bom}^4 \wedge d{\bom}^4 =
- {\bom}^1 \wedge {\bom}^3 \wedge {\bom}^4 \wedge {\bom}^1{}_4 - {\bom}^2 \wedge {\bom}^3 \wedge {\bom}^4 \wedge {\bom}^2{}_4.
\end{aligned}
\label{bigTheta}
\end{equation}
(Therefore, $\mathcal{I}$ is the ideal in $\Omega^*(\FX)$ generated,
algebraically, by the $4$-forms $\boldsymbol{\Theta}^i$ and the $5$-forms
$d\boldsymbol{\Theta}^i$.)
We consider the {\XDSIC} $(\mathcal{I}, \bOm)$
on the ten-dimensional manifold $\FX$, where the independence condition is
defined by the $4$-form
\[
\bOm := {\bom}^1 \wedge \dots \wedge {\bom}^4 \in \Omega^4(\FX).
\]

\vs As a result of the previous discussion, we have the following:
\begin{lemma}
Let $U \subseteq X$ is an open set, and $f \colon U \rightarrow \FX$ a section of
$\FX$ that satisfies $f^* {\bvar} = 0$, for all ${\bvar} \in \mathcal{I}$,
and $f^* \bOm \neq 0$ on $U$. Then the $1$-forms ${\bep}^a := f^* {\bom}^a
\in \Omega^1(U)$ define an orthonormal coframe on $U$ that
satisfies~\eqref{dep}.
\end{lemma}

Let $E_4 \subset T_p \FX$ be a $4$-dimensional integral element of
$(\mathcal{I}, \bOm)$ based at point $p \in \FX$
(i.e. $\restrict{\bvar}{E_4} = 0$, for all ${\bvar} \in \mathcal{I}$ and
$\restrict{\bOm}{E_4} \neq 0$). The space of such integral elements is
denoted by $V_4(\mathcal{I}, \bOm)$, and is a subset of
$\mathrm{Gr}_4(T\FX, \bOm)$, which is the subset of the Grassmannian bundle
$\mathrm{Gr}_4(T\FX)$ consisting of $4$-planes, $E_4$, for which
$\restrict{\bOm}{E_4} \neq 0$. Let $(\mathbf{v}_1, \dots, \mathbf{v}_4)$ be
a basis for $E_4$ which, without loss of generality, we may take to be of
the form
\begin{equation}
\mathbf{v}_a = \frac{\partial}{\partial {\bom}^a}(p)
+ \sum_{b < c} \llambda{a}{b}{c} \frac{\partial}{\partial {\bom}^b{}_c}(p),
\label{va}
\end{equation}
where $\{ \partial/\partial {\bom}^a, \partial / \partial {\bom}^a{}_b \}$
denotes the basis of $T\FX$ dual to $\{ {\bom}^a, {\bom}^a{}_b \}$ (see,
e.g.,~\cite[pp.~253]{Olver} for a discussion of this notation.). Note that
the coordinates $(x, g)$ on $\FX$ along with the parameters
$\left\{ \llambda{a}{b}{c}: a, b, c = 1, \dots, 4; \ b < c \right\}$, give a local
coordinate system on $\mathrm{Gr}_4(T\FX, \bOm)$. The condition that $E_4$
be an integral element of $\mathcal{I}$ is that
$\boldsymbol{\Theta}^i(\mathbf{v}_1, \mathbf{v}_2, \mathbf{v}_3,
\mathbf{v}_4) = 0$ for $i = 1, \dots, 4$. Substituting~\eqref{va}
into~\eqref{bigTheta}, we find that this is equivalent to the conditions
\begin{equation}
\begin{aligned}
\llambda{1}{2}{3} &= \llambda{2}{1}{3}, & \llambda{1}{2}{4} &= \llambda{2}{1}{4},
\\
\llambda{3}{1}{4} &= \llambda{4}{1}{3}, & \llambda{3}{2}{4} &= \llambda{4}{2}{3}.
\end{aligned}
\label{LambdaConditions}
\end{equation}
At each point $p \in \FX$, these equations impose $4$ linear constraints on
the coordinates $\llambda{a}{b}{c}$. It therefore follows that
$V_4(\mathcal{I}, \bOm)$ is a smooth submanifold of $Gr_4(T\FX)$ of
codimension $4$.

We now consider an integral flag $(0)_p \subset E_1 \subset E_2 \subset E_3
\subset E_4 \subset T_p(\FX)$, and wish to calculate the integers $c_k, k =
0, \dots, 4$ (see Definition~\ref{defn:A.3} in Appendix~\ref{app:XDS}). Since $\mathcal{I}$ contains no non-zero
$1$-forms, $2$-forms or $3$-forms, it follows that
\[
c_0 = c_1 = c_2 = 0
\]
and, from its definition, we have
\[
c_4 = \dim \FX - 4 = 6.
\]
Therefore, it only remains to calculate $c_3$. To do this, we first define
the one-forms
\[
{\bpi}^a{}_b := {\bom}^a{}_b(p) - \sum_c \llambda{c}{a}{b} \, {\bom}^c(p) \in T^*_p \FX.
\]
Note that the ${\bpi}^a{}_b$, with $a<b$, span the subspace of $T^*_p(\FX)$
that annihilate the vectors $\mathbf{v}_a$. It then follows that $E_4$ may
be described as
\[
E_4 = \left\{ \mathbf{v} \in T_p \FX: \vphantom{|^|} {\bpi}^a{}_b(\mathbf{v}) = 0, \mbox{ for } a, b = 1, \dots, 4;\ a < b \right\}.
\]
We now note that, by~\eqref{LambdaConditions}, we may write
\begin{align*}
\boldsymbol{\Theta}^1 &= {\bom}^1 \wedge {\bom}^2 \wedge {\bom}^3 \wedge {\bpi}^1{}_3 + {\bom}^1 \wedge {\bom}^2 \wedge {\bom}^4 \wedge {\bpi}^1{}_4,
\\
\boldsymbol{\Theta}^2 &= {\bom}^1 \wedge {\bom}^2 \wedge {\bom}^3 \wedge {\bpi}^2{}_3 + {\bom}^1 \wedge {\bom}^2 \wedge {\bom}^4 \wedge {\bpi}^2{}_4 ,
\\
\boldsymbol{\Theta}^3 &= - {\bom}^1 \wedge {\bom}^3 \wedge {\bom}^4 \wedge {\bpi}^1{}_3 - {\bom}^2 \wedge {\bom}^3 \wedge {\bom}^4 \wedge {\bpi}^2{}_3,
\\
\boldsymbol{\Theta}^4 &= - {\bom}^1 \wedge {\bom}^3 \wedge {\bom}^4 \wedge {\bpi}^1{}_4 - {\bom}^2 \wedge {\bom}^3 \wedge {\bom}^4 \wedge {\bpi}^2{}_4 .
\end{align*}
We let $E_3 := \mathrm{span}\,\{ \mathbf{e}_1, \mathbf{e}_2, \mathbf{e}_3 \} \subset E_4$, where
\[
\mathbf{e}_i = \sum_{a=1}^4 e_i^a \mathbf{v}_a, \qquad i = 1, 2, 3,
\]
and define the quantities
\begin{align*}
A &:=
\left( {\bom}^1 \wedge {\bom}^2 \wedge {\bom}^3 \right) \left(
\mathbf{e}_1, \mathbf{e}_2, \mathbf{e}_3 \right), \qquad
B &:=
\left( {\bom}^1 \wedge {\bom}^2 \wedge {\bom}^4 \right) \left( \mathbf{e}_1, \mathbf{e}_2, \mathbf{e}_3 \right),
\\
C &:=
\left( {\bom}^2 \wedge {\bom}^3 \wedge {\bom}^4 \right) \left(
\mathbf{e}_1, \mathbf{e}_2, \mathbf{e}_3 \right), \qquad
D &:=
\left( {\bom}^1 \wedge {\bom}^3 \wedge {\bom}^4 \right) \left( \mathbf{e}_1, \mathbf{e}_2, \mathbf{e}_3 \right).
\end{align*}
We then wish to consider the polar space
\[
H(E_3) := \left\{ \mathbf{v} \in T_p \FX: \vphantom{|^|}
{\bvar}(\mathbf{v}, \mathbf{e}_1, \mathbf{e}_2, \mathbf{e}_3) = 0,
\forall {\bvar} \in \mathcal{I} \right\}
\]
(see~Definition~\ref{defn:polarspace} in Appendix~\ref{app:XDS}). It follows that $\mathbf{v} \in T_p \FX$ lies in
$H(E_3)$ if and only if
\begin{equation}
\begin{aligned}
\boldsymbol{\Theta}^1 \left( \mathbf{v}, \mathbf{e}_1, \mathbf{e}_2, \mathbf{e}_3 \right)
&= - A {\bpi}^1{}_3(\mathbf{v}) - B {\bpi}^1{}_4(\mathbf{v}) = 0,
\\
\boldsymbol{\Theta}^2 \left( \mathbf{v}, \mathbf{e}_1, \mathbf{e}_2, \mathbf{e}_3 \right)
&= - A {\bpi}^2{}_3(\mathbf{v}) - B {\bpi}^2{}_4(\mathbf{v}) = 0,
\\
\boldsymbol{\Theta}^3 \left( \mathbf{v}, \mathbf{e}_1, \mathbf{e}_2, \mathbf{e}_3 \right)
&= D {\bpi}^1{}_3(\mathbf{v}) + C {\bpi}^2{}_3(\mathbf{v}) = 0,
\\
\boldsymbol{\Theta}^4 \left( \mathbf{v}, \mathbf{e}_1, \mathbf{e}_2, \mathbf{e}_3 \right)
&= D {\bpi}^1{}_4(\mathbf{v}) + C {\bpi}^2{}_4(\mathbf{v}) = 0.
\end{aligned}
\label{linearconstraints}
\end{equation}
Since ${\bpi}^1{}_3, {\bpi}^1{}_4, {\bpi}^2{}_3, {\bpi}^2{}_4$ are
linearly-independent $1$-forms on $\FX$, it follows that the number of
linearly-independent constraints imposed on a vector $\mathbf{v} \in T_p
\FX$ by equations~\eqref{linearconstraints} is equal to the rank of the
matrix
\[
\alpha :=\left( \begin{array}{c c c c}
-A &-B &0 &0\\ 0 &0 &-A &-B \\ D &0 &C &0 \\ 0 &D &0 &C
\end{array}\right).
\]
Since $\det \alpha = 0$, it follows that $c_3 \le 3$. Any flag $(0)_p
\subset E_1 \subset E_2 \subset E_3 \subset E_4$ such that $\rank \alpha =
3$ (e.g. $A = C = 1$, $B = D = 0$) will give rise to $3$
linearly-independent $1$-forms, $({\bpi}^1, {\bpi}^2, {\bpi}^3)$, such that
$H(E_3) = \left\{ \mathbf{v} \in T_p \FX: \vphantom{|^|} {\bpi}^1(\mathbf{v}) =
{\bpi}^2(\mathbf{v}) = {\bpi}^3(\mathbf{v}) = 0 \right\}$. Hence $c_3 = 3$ for such
an integral flag.

\begin{corollary}
The {\XDSIC} $(\mathcal{I}, \bOm)$ contains no integral elements of
dimension $4$ that pass Cartan's test.
\end{corollary}
\begin{proof}
The codimension of $V_4(\mathcal{I}, \bOm)$ at any integral element is
equal to $4$. Any four-dimensional integral flag has $c_0 = c_1 = c_2 = 0$
and $c_3 \le 3$. Therefore $c_0 + c_1 + c_2 + c_3 \le 3 \neq 4$, so no such
integral element passes Cartan's test.
\end{proof}

Note that the non-maximality of the rank of $\alpha$ is essentially the
same algebraic condition that led to the non-existence of
non-characteristic surfaces when we studied the linearisation of the
block-diagonalisation problem in Section~\ref{sec:block}.

\subsection{Prolongation}
Since the system $(\mathcal{I}, \bOm)$ is not involutive, we cannot
directly apply the Cartan--K\"{a}hler theorem. There is a standard technique
for dealing with such non-involutive exterior differential systems, namely
\emph{prolongation} (see, e.g.,~\cite{BCGGG, IL, Olver}). In the current
context, the (first) prolongation of the system $(\mathcal{I}, \bOm)$ is a
Pfaffian system defined on the manifold of four-dimensional integral
elements, $V_4(\mathcal{I}, \bOm)$, of the system $(\mathcal{I}, \bOm)$. In
particular, recall that $(x, g, \llambda{a}{b}{c})$ define a local
coordinate system on the Grassmannian bundle $Gr_4(T \FX)$ of four-planes
in the tangent bundle of $\FX$. Moreover, the space $M^{(1)} :=
V_4(\mathcal{I}, \bOm)$ is a thirty-dimensional manifold of the form $\FX
\times \mathbb{R}^{20}$, with the parameters $\llambda{a}{b}{c}$ subject to
the symmetry conditions~\eqref{LambdaConditions} as coordinates in the
$\mathbb{R}^{20}$ direction. (In particular, the conditions imposed by the
exterior differential system $(\mathcal{I}, \bOm)$ have already been imposed.)
As such $M^{(1)}$ may be viewed as a subspace
of the bundle $Gr_4(T \FX)$. The bundle $Gr_4(T \FX)$ comes equipped with a
natural set of contact forms, and the Pfaffian system that we consider on
$M^{(1)}$ is generated by the restriction of these differential forms.

More explicitly, we now consider the {\XDSIC}, $(\mathcal{I}^{(1)}, \bOm)$,
on the space $M^{(1)}$ generated by the $1$-forms
\begin{equation}
\bth^a{}_b := {\bom}^a{}_b - \sum_c \llambda{c}{a}{b} \, {\bom}^c, \qquad a, b = 1, \dots, 4; \qquad a < b,
\end{equation}
where ${\bom}^a{}_b$ and ${\bom}^a$ now denote the pull-backs to $M^{(1)}$
of the corresponding forms on $\FX$, with the independence condition
defined by the $4$-form $\bOm := {\bom}^1 \wedge {\bom}^2 \wedge {\bom}^3
\wedge {\bom}^4$.

We now look for four-dimensional integral elements, $E_4 \in
V_4(\mathcal{I}^{(1)}, \bOm)$, of this system. The point is that if $U$ is
an open subset of $X$ and $f \colon U \rightarrow M^{(1)}$ a local section of the
bundle $M^{(1)}$ with the property that $f^* \bth^a{}_b = 0$, $f^* \bOm
\neq 0$, then ${\bep}^i := f^* {\bom}^i$ define an orthonormal coframe on
$U$ that obeys~\eqref{dep}. As such, integral manifolds of
$(\mathcal{I}^{(1)}, \bOm)$ define solutions of our block-diagonalisation
problem. As a first step in showing the existence of such integral
manifolds we show that the system $(\mathcal{I}^{(1)}, \bOm)$ on $M^{(1)}$
is involutive. Applying the Cartan--K\"{a}hler theorem then gives the
solution to our block-diagonalisation problem. Our method here follows that
of~\cite{BCGGG}, Chapter~VII, \S3.

\vs
A short calculation shows that
\begin{equation}
d{\bth}^a{}_b \equiv - \sum_c d{\llambda{c}{a}{b}} \wedge {\bom}^c + \frac{1}{2} \sum_{c, d} T^a{}_{bcd} \, {\bom}^c \wedge {\bom}^d \mod{\bth},
\label{dtheta}
\end{equation}
where we have defined
\[
T^a{}_{bcd} := R^a{}_{bcd} + \sum_e \left[ \llambda{e}{a}{b} \left( \llambda{c}{e}{d} - \llambda{d}{e}{c} \right) - \llambda{c}{a}{e} \llambda{d}{e}{b} + \llambda{d}{a}{e} \llambda{c}{e}{b} \right].
\]
The second term in Equation~\eqref{dtheta} implies that there is torsion in
the Pfaffian system. We would like to absorb the torsion terms by
writing~\eqref{dtheta} in the form $d{\bth}^a{}_b \equiv - \sum_c
\boldpi{c}{a}{b} \wedge {\bom}^c \mod{\bth}$, where $\boldpi{c}{a}{b}
\equiv d{\llambda{c}{a}{b}} \mod{\bom^i}$ and the $1$-forms
${\bpi}_a{}^b{}_c$, $a, b, c = 1, \dots, 4$, $b < c$ obey symmetry
relations analogous to~\eqref{LambdaConditions} (e.g. $\boldpi{1}{2}{3} =
\boldpi{2}{1}{3}$). However, in the present case, there is an obstruction
to the existence of such $1$-forms $\boldpi{a}{b}{c}$, which lies in the
quantity, $T(x, g, \lambda)$, defined by the relation
\begin{equation}
\begin{aligned}
{\bom}^1 \wedge {\bom}^3 \wedge d{\bth}^1{}_3 + {\bom}^1 \wedge {\bom}^4 \wedge d{\bth}^1{}_4 +&
{\bom}^2 \wedge {\bom}^3 \wedge d{\bth}^2{}_3 + {\bom}^2 \wedge {\bom}^4 \wedge d{\bth}^2{}_4
\\
&\equiv
- 2 T(x, g, \lambda) \, {\bom}^1 \wedge {\bom}^2 \wedge{\bom}^3 \wedge {\bom}^4 \mod{\bth}.
\end{aligned}
\label{Tnot0}
\end{equation}
$T(x,g,\lambda)$ is then given in terms of the curvature by the expression
\begin{align*}
T(x, g, \lambda) := &R_{1234}(x, g) + \llambda{1}{2}{3} \left( \llambda{2}{2}{4} - \llambda{1}{1}{4} \right)
+ \llambda{1}{2}{4} \left( \llambda{1}{1}{3} - \llambda{2}{2}{3} \right)
\\
&\hskip 3cm
+ \llambda{3}{4}{1} \left( \llambda{4}{4}{2} - \llambda{3}{3}{2} \right)
+ \llambda{3}{4}{2} \left( \llambda{3}{3}{1} - \llambda{4}{4}{1} \right).
\end{align*}
In particular, an explicit calculation (for details,
see Appendix~\ref{sec:absorb}) shows that it is possible to absorb most of the
torsion terms in \eqref{Tnot0} and there exist $1$-forms,
$\boldpi{a}{b}{c}$, on $M^{(1)}$ satisfying $\boldpi{a}{b}{c} \equiv
d{\llambda{a}{b}{c}} \mod{{\bom}^i}$ in terms of which
equations~\eqref{dtheta} take the form
\begin{equation}
\begin{aligned}
d{\bth}^1{}_2 &\equiv - \sum_a \boldpi{a}{1}{2} \wedge {\bom}^a \mod{\bth},
\\
d{\bth}^1{}_3 &\equiv - \sum_a \boldpi{a}{1}{3} \wedge {\bom}^a \mod{\bth},
\\
d{\bth}^1{}_4 &\equiv - \sum_a \boldpi{a}{1}{4} \wedge {\bom}^a \mod{\bth},
\\
d{\bth}^2{}_3 &\equiv - \sum_a \boldpi{a}{2}{3} \wedge {\bom}^a \mod{\bth},
\\
d{\bth}^2{}_4 &\equiv - \sum_a \boldpi{a}{2}{4} \wedge {\bom}^a + 2 T {\bom}^1 \wedge {\bom}^3 \mod{\bth},
\\
d{\bth}^3{}_4 &\equiv - \sum_a \boldpi{a}{3}{4} \wedge {\bom}^a \mod{\bth}.
\end{aligned}
\label{essentialtorsion}
\end{equation}
Equation~\eqref{Tnot0} implies, however, that it is not possible to absorb
the remaining torsion by a redefinition of the forms $\boldpi{a}{b}{c}$. In
particular, it implies that there is essential torsion in the system
characterised by the function $T$. The existence of such essential torsion
implies that a necessary condition for the existence of an integral element
$E_4 \subset T_p M^{(1)}$ based at a point $p \in M^{(1)}$ is that $p$
satisfies the compatibility condition $T(p) = 0$. We define the
non-singular part of the subspace where this condition holds,
\[
S^{(1)} := \left\{p \in M^{(1)}: \vphantom{|^|} T(p) = 0, dT(p) \neq 0 \right\},
\]
which (by the implicit function theorem) is a codimension-one submanifold,
$i \colon S \rightarrow M^{(1)}$, of $M^{(1)}$.

\begin{remark}
Note that an explicit calculation of $dT$ shows that, given $(x, g) \in
\FX$, for generic $\lambda$ we have $dT(x, g, \lambda) \neq 0$.
\end{remark}

We define the $1$-forms ${\widetilde{\bth}}{}^a{}_b := i^* {\bth}^a{}_b$,
$\widetilde{\bom}^a := i^* {\bom}^a$ on $S$, and consider the Pfaffian
system $(\widetilde{\mathcal{I}}, \widetilde{\Omega})$ on $S$ generated by
$\{ {\widetilde{\bth}}{}^a{}_b \}$ with independence condition
$\widetilde{\Omega} := i^* \Omega = \widetilde{\bom}^1 \wedge
\widetilde{\bom}^2 \wedge \widetilde{\bom}^3 \wedge \widetilde{\bom}^4$. We
then have the following:

\begin{proposition}
There exist $1$-forms, $\boldpitilde{a}{b}{c} \in \Omega^1(S)$, for $a, b,
c = 1, \dots, 4$ with $b < c$, that satisfy
\begin{enumerate}
\item $\boldpitilde{a}{b}{c} \equiv i^* \!\left( d{\llambda{a}{b}{c}}
\right) \! \mod{\widetilde{\bom}^i}$,
\item $\boldpitilde{1}{2}{3} = \boldpitilde{2}{1}{3}, \quad
\boldpitilde{1}{2}{4} = \boldpitilde{2}{1}{4}, \quad
\boldpitilde{3}{1}{4} = \boldpitilde{4}{1}{3}, \quad
\boldpitilde{3}{2}{4} = \boldpitilde{4}{2}{3}$,
\end{enumerate}
with the property that
\begin{equation}
d{\widetilde{\bth}}{}^a{}_b \equiv - \sum_c \boldpitilde{c}{a}{b} \wedge \widetilde{\bom}^c \mod{\widetilde{\bth}}.
\label{dbth}
\end{equation}
\label{prop:dbth}
\end{proposition}
\begin{proof}
Taking the pull-back of equations~\eqref{essentialtorsion} to $S$, and
using the fact that $T \circ i = 0$, we deduce that the $1$-forms
$\boldpitilde{a}{b}{c} := i^* \!\left( \boldpi{a}{b}{c} \right)$ on $S$
have the required properties.
\end{proof}

Rather than using $\llambda{a}{b}{c}$ as coordinates it will be useful to
introduce new coordinates $y^1,\dots, y^8$ and $z^1,\dots, z^4$ on
$M^{(1)}$ defined by
\begin{align*}
y^1 &:= \llambda{1}{2}{3}, & &y^2 := \frac{1}{2} \left( \llambda{2}{2}{4} - \llambda{1}{1}{4} \right),
&y^3 &:= \llambda{1}{2}{4}, & &y^4 := \frac{1}{2} \left( \llambda{2}{2}{3} - \llambda{1}{1}{3} \right),
\\
y^5 &:= \llambda{3}{4}{1}, & &y^6 := \frac{1}{2} \left( \llambda{4}{4}{2} - \llambda{3}{3}{2} \right),
&y^7 &:= \llambda{3}{4}{2}, & &y^8 := \frac{1}{2} \left( \llambda{4}{4}{1} - \llambda{3}{3}{1} \right)
\end{align*}
and
\begin{align*}
z^1 &:= \frac{1}{2} \left( \llambda{2}{2}{4} + \llambda{1}{1}{4} \right), \quad
& &z^2 := \frac{1}{2} \left( \llambda{2}{2}{3} + \llambda{1}{1}{3} \right), \quad
\\
z^3 &:= \frac{1}{2} \left( \llambda{4}{4}{2} + \llambda{3}{3}{2} \right), \quad
& &z^4 := \frac{1}{2} \left( \llambda{4}{4}{1} + \llambda{3}{3}{1} \right).
\end{align*}
In terms of these coordinates our constraint equation takes the form
\begin{equation}
T(x, g, y, z) = R_{1234}(x, g) + 2 \left( y^1 y^2 - y^3 y^4 + y^5 y^6 - y^7 y^8 \right) = 0,
\label{constraint1}
\end{equation}
so that the constraint does not depend upon the $z$ coordinates.

We now write the structure equations~\eqref{dbth} in the form
\begin{equation}
d\left(\begin{array}{c}
\widetilde{\bth}{}^1{}_2
\\
\widetilde{\bth}{}^1{}_3
\\
\widetilde{\bth}{}^1{}_4
\\
\widetilde{\bth}{}^2{}_3
\\
\widetilde{\bth}{}^2{}_4
\\
\widetilde{\bth}{}^3{}_4
\end{array}\right)
\equiv \pi \wedge
\left(\begin{array}{c}
\widetilde{\bom}^1
\\
\widetilde{\bom}^2
\\
\widetilde{\bom}^3
\\
\widetilde{\bom}^4
\end{array}\right) \!
\mod{\widetilde{\bth}}.
\label{structureequations}
\end{equation}
Here, the matrix of $1$-forms $\pi$ (which, modulo $\{ \widetilde{\bth},
\widetilde{\bom} \}$, is the tableau matrix of $(\widetilde{\mathcal{I}},
\widetilde{\bOm})$ at $x$) is given by
\begin{equation}
\pi =
-\left(\begin{array}{c c c c}
\boldpitilde{1}{1}{2} &\boldpitilde{2}{1}{2} &\boldpitilde{3}{1}{2} &\boldpitilde{4}{1}{2}
\\
\boldpitilde{1}{1}{3} &\boldpitilde{2}{1}{3} &\boldpitilde{3}{1}{3} &\boldpitilde{4}{1}{3}
\\
\boldpitilde{1}{1}{4} &\boldpitilde{2}{1}{4} &\boldpitilde{4}{1}{3} &\boldpitilde{4}{1}{4}
\\
\boldpitilde{2}{1}{3} &\boldpitilde{2}{2}{3} &\boldpitilde{3}{2}{3} &\boldpitilde{4}{2}{3}
\\
\boldpitilde{2}{1}{4} &\boldpitilde{2}{2}{4} &\boldpitilde{4}{2}{3} &\boldpitilde{4}{2}{4}
\\
\boldpitilde{1}{3}{4} &\boldpitilde{2}{3}{4} &\boldpitilde{3}{3}{4} &\boldpitilde{4}{3}{4}
\end{array}\right)
\label{tableau}
\end{equation}
In order to simplify notation, we define the $1$-forms $\widetilde{\bpi}^{\alpha}$, $\alpha = 1, \dots, 8$, by
\begin{align*}
\widetilde{\bpi}^1 &:= \boldpitilde{1}{2}{3}, \qquad
&\widetilde{\bpi}^2 &:= \frac{1}{2} \left( \boldpitilde{2}{2}{4} - \boldpitilde{1}{1}{4} \right),
\\
\widetilde{\bpi}^3 &:= \boldpitilde{1}{2}{4}, \qquad
&\widetilde{\bpi}^4 &:= \frac{1}{2} \left( \boldpitilde{2}{2}{3} - \boldpitilde{1}{1}{3} \right),
\\
\widetilde{\bpi}^5 &:= \boldpitilde{3}{4}{1}, \qquad
&\widetilde{\bpi}^6 &:= \frac{1}{2} \left( \boldpitilde{4}{4}{2} - \boldpitilde{3}{3}{2} \right),
\\
\widetilde{\bpi}^7 &:= \boldpitilde{3}{4}{2}, \qquad
&\widetilde{\bpi}^8 &:= \frac{1}{2} \left( \boldpitilde{4}{4}{1} - \boldpitilde{3}{3}{1} \right),
\end{align*}
which have the property that
$\widetilde{\bpi}^{\alpha} \equiv i^*dy^{\alpha} \mod{\widetilde{\bom}^a}$ for $\alpha = 1, \dots, 8$.
We also define the 1-forms $\widetilde{\bro^a}$, $a=1,\dots, 4$ by
\begin{align*}
\widetilde{\bro}^1 &:= \frac{1}{2} \left( \boldpitilde{2}{2}{4}
+\boldpitilde{1}{1}{4} \right),
&\widetilde{\bro}^2 &:= \frac{1}{2} \left( \boldpitilde{2}{2}{3} + \boldpitilde{1}{1}{3} \right),
\\
\widetilde{\bro}^3 &:= \frac{1}{2} \left(
\boldpitilde{4}{4}{2}+\boldpitilde{3}{3}{2} \right),
&\widetilde{\bro}^4 &:= \frac{1}{2} \left( \boldpitilde{4}{4}{1} + \boldpitilde{3}{3}{1} \right),
\end{align*}
which have the property that $\widetilde{\bro}^a \equiv i^* dz^a
\pmod{\widetilde{\bom}^1,\dots, \widetilde{\bom}^4}$ for $a = 1, \dots, 4$.
In addition, we define $1$-forms $\{ \widetilde{\bmu}^a \}_{a=1}^4$ and
$\{\widetilde{\bnu}^a \}_{a=1}^4$ by
\begin{align*}
\left( \widetilde{\bmu}^1, \dots, \widetilde{\bmu}^4 \right) &:=
\left( \boldpitilde{1}{1}{2}, \boldpitilde{2}{1}{2}, \boldpitilde{3}{1}{2}, \boldpitilde{4}{1}{2} \right),
\\
\left( \widetilde{\bnu}^1, \dots, \widetilde{\bnu}^4 \right) &:=
\left( \boldpitilde{1}{3}{4}, \boldpitilde{2}{3}{4}, \boldpitilde{3}{3}{4}, \boldpitilde{4}{3}{4} \right).
\end{align*}
In this notation we have
\begin{equation}
\pi =
-\left(\begin{array}{c c c c}
\widetilde{\bmu}^1 &\widetilde{\bmu}^2 &\widetilde{\bmu}^3 &\widetilde{\bmu}^4
\\
\widetilde{\bro}^2 - \widetilde{\bpi}{}^4 &\widetilde{\bpi}{}^1 &-\widetilde{\bro}^4 + \widetilde{\bpi}{}^8 &\widetilde{\bpi}{}^5
\\
\widetilde{\bro}^1 - \widetilde{\bpi}{}^2 &\widetilde{\bpi}{}^3 &\widetilde{\bpi}{}^5 &-\widetilde{\bro}^4 - \widetilde{\bpi}{}^8
\\
\widetilde{\bpi}{}^1 & \widetilde{\bro}^2 + \widetilde{\bpi}{}^4 &-\widetilde{\bro}^3 + \widetilde{\bpi}{}^6 &\widetilde{\bpi}{}^7
\\
\widetilde{\bpi}{}^3 &\widetilde{\bro}^1 + \widetilde{\bpi}{}^2 &\widetilde{\bpi}{}^7 &-\widetilde{\bro}^3 - \widetilde{\bpi}{}^6
\\
\widetilde{\bnu}^1 &\widetilde{\bnu}^2 &\widetilde{\bnu}^3 &\widetilde{\bnu}^4
\end{array}\right).
\label{tableau2}
\end{equation}

We now note that, since the functions $(x, g, y)$ obey
equation~\eqref{constraint1} on $S^{(1)}$, they will not be functionally
independent when pulled back to $S$. In particular, we require that $i^*
(dT) = 0$, which translates into the condition that
\[
2 \left( \widetilde{y}^1 d\widetilde{y}^2 + \widetilde{y}^2 d\widetilde{y}^1 - \widetilde{y}^3 d\widetilde{y}^4 - \widetilde{y}^4 d\widetilde{y}^3 + \widetilde{y}^5 d\widetilde{y}^6 + \widetilde{y}^6 d\widetilde{y}^5 - \widetilde{y}^7 d\widetilde{y}^8 - \widetilde{y}^8 d\widetilde{y}^7 \right) + \sum_a \Phi_a \widetilde{\bom}^a \equiv 0 \mod{ {\widetilde{\bth}}}
\]
on $S$, where
\[
\Phi_a = i^* \left( \frac{\partial}{\partial {\bom}^a} R_{1234} \right) +
\sum_b \left[ \llambdatilde{a}{b}{1} \widetilde{R}_{b234} +
\llambdatilde{a}{b}{2} \widetilde{R}_{1b34} + \llambdatilde{a}{b}{3}
\widetilde{R}_{12b4} + \llambdatilde{a}{b}{4} \widetilde{R}_{123b} \right],
\]
and $\widetilde{y}^{\alpha} := i^* y^{\alpha} = y^{\alpha} \circ i$, etc,
denote the pull-backs to $S$ of the corresponding functions on
$M^{(1)}$. In particular, since $\widetilde{\bpi}^{\alpha} \equiv
i^*dy^{\alpha} \mod{\widetilde{\bom}^a}$ for $\alpha = 1, \dots, 8$, there
exist functions $\Psi_a$ on $S$ such that
\begin{equation}
i^* (dT) = \widetilde{y}^1 \widetilde{\bpi}{}^2 + \widetilde{y}^2 \widetilde{\bpi}{}^1 - \widetilde{y}^3 \widetilde{\bpi}{}^4 - \widetilde{y}^4 \widetilde{\bpi}{}^3 + \widetilde{y}^5 \widetilde{\bpi}{}^6 + \widetilde{y}^6 \widetilde{\bpi}{}^5 - \widetilde{y}^7 \widetilde{\bpi}{}^8 - \widetilde{y}^8 \widetilde{\bpi}{}^7 + \sum_a \Psi_a \widetilde{\bom}^a \equiv 0 \mod{ {\widetilde{\bth}}}.
\label{lindep}
\end{equation}
Recall, however, that the $1$-forms, $\boldpitilde{a}{b}{c}$, are not
uniquely determined, and that we may add to them any linear combination of
the $1$-forms $\{ \widetilde{\bom}^a \}$ consistent with
equations~\eqref{structureequations} and~\eqref{tableau}. At a generic
point $p \in S$ at which $\widetilde{y}^1(p), \dots, \widetilde{y}^8(p)$
are all non-zero, it is shown in Appendix~\ref{sec:absorb2} that all the
functions $\Psi_a$ in equation~\eqref{lindep} may be absorbed into a
redefinition of the $1$-forms $\widetilde{\bpi}{}^1, \dots,
\widetilde{\bpi}{}^8$ and $\widetilde{\bro}{}^1, \dots,
\widetilde{\bro}{}^4$. Noting that the non-vanishing of $\widetilde{y}^1,
\dots, \widetilde{y}^8$ is an open condition, we deduce that we may take
the $1$-forms $\widetilde{\bpi}{}^1, \dots, \widetilde{\bpi}{}^8$ to obey
the linear-dependence condition
\begin{equation}
\widetilde{y}^1 \widetilde{\bpi}{}^2 + \widetilde{y}^2 \widetilde{\bpi}{}^1 - \widetilde{y}^3 \widetilde{\bpi}{}^4 - \widetilde{y}^4 \widetilde{\bpi}{}^3 + \widetilde{y}^5 \widetilde{\bpi}{}^6 + \widetilde{y}^6 \widetilde{\bpi}{}^5 - \widetilde{y}^7 \widetilde{\bpi}{}^8 - \widetilde{y}^8 \widetilde{\bpi}{}^7 \equiv 0 \mod{ {\widetilde{\bth}}}
\label{lindep2}
\end{equation}
on an open neighbourhood, $U$, of the point $p$ in $S$. This relationship
implies (via the structure equations~\eqref{structureequations}) that the
essential torsion of the system $(\widetilde{\mathcal{I}},
\widetilde{\bOm})$ is zero on the open set $U$. It follows from
Proposition~\ref{5.16} that the system $(\widetilde{\mathcal{I}},
\widetilde{\bOm})$ is involutive at $p$ if and only if the tableau $A_p$ is
involutive.

To show that this is the case, we need to know the reduced Cartan
characters of the tableau $A_p$, and the dimension of the first
prolongation, $A_p^{(1)}$, of $A_p$.

\begin{proposition}
The first prolongation of the tableau $A_p$ is an affine-linear space of
dimension $41$.
\label{prop41}
\end{proposition}
\begin{proof}
See Appendix~\ref{sec:absorb2}.
\end{proof}

\vs

\begin{proposition}
The system $(\widetilde{\mathcal{I}}, \widetilde{\Omega})$ has reduced
Cartan characters
\[
s_1^{\prime} = 6, \quad s_2^{\prime} = 6, \quad s_3^{\prime} = 5, \quad s_4^{\prime} = 2.
\]
\end{proposition}
\begin{proof}
Let $p \in S$ with $\widetilde{y}^1(p), \dots, \widetilde{y}^8(p)$ all
non-zero. Equation~\eqref{lindep} may then be looked on as defining one of
the $1$-forms, say ${\widetilde{\bpi}}^8$, in terms of the other
seven. Note that, since the thirty $1$-forms $\{ \widetilde{\bom}^i,
\widetilde{\bth}{}^a{}_b, {\widetilde{\bro}}^a, {\widetilde{\bmu}}^a,
{\widetilde{\bnu}}^a \}$ must span the cotangent space at each point of the
twenty-nine-dimensional manifold $S$, it follows that the linear
relation~\eqref{lindep} is the only relation obeyed by these $1$-forms on
$S$. As such, once we have substituted for ${\widetilde{\bpi}}^8$, say, the
remaining differential forms $\{ \widetilde{\bpi}{}^1, \dots,
\widetilde{\bpi}{}^7, {\widetilde{\bro}}^a, {\widetilde{\bmu}}^a,
{\widetilde{\bnu}}^a \}$ that appear in the matrix $\pi$ are
linearly-independent on $S$.

We then consider the tableau matrix, $\overline{\pi} := \pi
\mod{\widetilde{\bth}, \widetilde{\bom}}$, and we wish to calculate the
reduced Cartan characters. This should be computed with respect to a
generic basis of $1$-forms $\{ {\bom}^i \}$, so we note that the tableau
relative to a different basis, $\underline{\widetilde{\bom}}^a := \sum_b
\left( \sigma^{-1} \right)^a{}_b \, {\widetilde{\bom}}^b$ where $\sigma \in
\mathrm{GL}(4, \mathbb{R})$, is given by $\overline{\pi}_{\sigma} :=
\overline{\pi} \sigma$. Substituting for ${\widetilde{\bpi}}^8$ into the
tableau matrix and noting that this is the only relationship that our
differential forms obey, we see that $\pi$ then has six
linearly-independent $1$-forms in its first column:
\[
{\widetilde{\bmu}}{}^1, \quad {\widetilde{\bro}}^2 - {\widetilde{\bpi}}^4, \quad {\widetilde{\bro}}^1 - {\widetilde{\bpi}}^2, \quad {\widetilde{\bpi}}^1, \quad {\widetilde{\bpi}}^3, \quad {\widetilde{\bnu}}{}^1.
\]
Therefore $s_1^{\prime} = 6$. The $1$-forms in column three:
\[
{\widetilde{\bmu}}{}^3, \quad -{\widetilde{\bro}}^4 + {\widetilde{\bpi}}^8, \quad {\widetilde{\bpi}}^5, \quad -{\widetilde{\bro}}^3 + \quad {\widetilde{\bpi}}^6, \quad {\widetilde{\bpi}}^7, \quad {\widetilde{\bnu}}{}^3
\]
are then linearly-independent, and independent of those in column one. (In
the preceding equation, we substitute for ${\widetilde{\bpi}}^8$ using
equation~\eqref{lindep2}.) Therefore $s_2^{\prime} = 6$. If we then
consider the linear combination of $\alpha$ times column two and $\beta$
times column four of~\eqref{tableau}, we gain the $1$-forms
\[
\alpha {\widetilde{\bmu}}^2 + \beta {\widetilde{\bmu}}^4,
\quad \alpha {\widetilde{\bpi}}^3 - \beta ( {\widetilde{\bro}}^4 + {\widetilde{\bpi}}^8 ),
\quad \alpha ( {\widetilde{\bro}}^2 + {\widetilde{\bpi}}{}^4 ) + \beta {\widetilde{\bpi}}^7,
\quad \alpha ( {\widetilde{\bro}}^1 + {\widetilde{\bpi}}^2 ) - \beta ( {\widetilde{\bro}}^3 + {\widetilde{\bpi}}^6 ),
\quad \alpha {\widetilde{\bnu}}^2 + \beta {\widetilde{\bnu}}^4.
\]
If we then take $\alpha, \beta$ both non-zero, this gives five more
linearly-independent $1$-forms. Therefore $s_3^{\prime} = 5$. Finally,
$s_1^{\prime} + s_2^{\prime} + s_3^{\prime} + s_4^{\prime} = 19$, the
number of linearly-independent $1$-forms in $\pi$, which fixes
$s_4^{\prime} = 2$.

Note that the above is equivalent to taking
\[
\sigma = \left(\begin{array}{c c c c} 1 &0 &0 &* \\ 0 &0 &\alpha &* \\ 0 &1 &0 &* \\ 0 &0 &\beta &* \end{array}\right),
\]
where the last column is only constrained by the requirement that $\sigma$
be non-singular.
\end{proof}

\begin{proposition}
The Pfaffian differential system $(\widetilde{\mathcal{I}},
\widetilde{\bOm})$ is involutive at $p$.
\end{proposition}
\begin{proof}
\[
s_1^{\prime} + 2 s_2^{\prime} + 3 s_3^{\prime} + 4 s_4^{\prime} = 6 + 12 + 15 + 8 = 41 = \dim A_p^{(1)}.
\]
\end{proof}

\begin{theorem}
Let $X$ be an analytic manifold, and $\mathbf{g}$ an analytic Riemannian
metric on $X$. For each $x \in X$, there exists a neighbourhood of $x$ on
which there exists an analytic coordinate system in terms of which the
metric $\mathbf{g}$ takes block-diagonal form.
\end{theorem}
\begin{proof}
Given any point $x \in M$, choose a generic point $p \in \pi^{-1}(x) \in
S$. By the previous Proposition, the system $(\widetilde{\mathcal{I}},
\widetilde{\bOm})$ is involutive. Applying the Cartan--K\"{a}hler theorem
(cf. Remark~\ref{PfaffianCK}), we deduce that there exists an integral
manifold of the {\XDSIC} $(\widetilde{\mathcal{I}}, \widetilde{\Omega})$
through $p$. This integral manifold corresponds to a section $f \colon X
\rightarrow S$ and hence to an orthonormal coframe $\{ {\bep}^i \}$ on a
neighbourhood of $x$ that obeys equation~\eqref{dep}.
\end{proof}

\begin{remark}
The solution to~\eqref{dep} is not unique but one has the freedom to independently make rotations in the $(\bep^1, \bep^2)$ and $(\bep^3, \bep^4)$ planes (equivalently, in the $(t, x)$ and $(y, z)$ planes of the proof of Proposition~\ref{prop:1}). This corresponds to the freedom to make rotations in the $({\bom}^1,{\bom}^2)$ and $({\bom}^3, {\bom}^4)$ planes without changing $(\mathcal{I}, \bOm)$. As a result the characteristic manifold is parameterised by two functions of four variables, consistent with the result that $s'_4 = 2$.
\end{remark}

\begin{remark}
The coordinate functions $\llambda{a}{b}{c}$ pull back to define functions
on $X$ that give the components, $\{ \boldsymbol{\Gamma}^a{}_b \}$, of the
Levi-Civita connection of the coframe $\{ {\bep}^a \}$. The curvature of
$\boldsymbol{\Gamma}$, $\mathbf{R}^{\boldsymbol{\Gamma}}$, then
automatically obeys the condition that
\begin{align}
R^{\boldsymbol{\Gamma}}_{1234} + \GGamma{1}{2}{3} \left( \GGamma{2}{2}{4} - \GGamma{1}{1}{4} \right)
&+ \GGamma{1}{2}{4} \left( \GGamma{1}{1}{3} - \GGamma{2}{2}{3} \right)
\nonumber
\\
&+ \GGamma{3}{4}{1} \left( \GGamma{4}{4}{2} - \GGamma{3}{3}{2} \right)
+ \GGamma{3}{4}{2} \left( \GGamma{3}{3}{1} - \GGamma{4}{4}{1} \right) = 0.
\label{curvcondition}
\end{align}
In the present context, this condition is derived from pulling back the
condition $T(p) = 0$ that was required for our Pfaffian system on $S$ to
have integral elements. However, it can also be shown that this condition
arises directly from the symmetry requirements on the Levi-Civita
connection (analogous to~\eqref{LambdaConditions}) that follow from
imposing~\eqref{dep}.

It turns out that~\eqref{curvcondition} has a simple geometrical
interpretation. Let $R^{\perp}_{1234}$ denote the curvature of the
connection of the bundle normal to the ${\bep}^1\wedge{\bep}^2$
plane. This is related to the full curvature and the associated
fundamental form $A_{\mathbf{U}}$ by the Ricci equation
\[
\mathbf{g} \left( \mathbf{R}^{\perp}(\mathbf{X}, \mathbf{Y})\mathbf{V}, \mathbf{U} \right) =
\mathbf{g} \left( \mathbf{R}(\mathbf{X}, \mathbf{Y})\mathbf{V}, \mathbf{U} \right)
- \mathbf{g} \left( \left[ A_{\mathbf{U}}, A_{\mathbf{V}} \right] \mathbf{X}, \mathbf{Y} \right).
\]
In the same way one can use the Ricci equation to obtain an
expression for the curvature $\tilde R^{\perp}_{3412}$ of the connection of
the bundle normal
to the ${\bep}^3\wedge{\bep}^4$ plane. Then by adding the expressions for the
two normal curvatures together one may write the curvature
condition~\eqref{curvcondition} in the alternative form
\begin{equation}
R^{\perp}_{1234}+\tilde R^{\perp}_{3412}=R_{1234}.
\label{curvcond}
\end{equation}
So that the full curvature is just the sum of the two normal curvatures.
\end{remark}

\subsection{The Lorentzian case}
Although we have carried out all of our calculations for the case of a
Riemannian four-manifold, the calculations carry through, essentially
unchanged, if the metric has Lorentzian signature. We can easily obtain the
geometric condition corresponding to~\eqref{curvcond} by using the
Newman--Penrose null formalism (see e.g.~\cite{PR1}). We start by
introducing a (complex) basis of null 1-forms
$(\boldsymbol{\ell}, \mathbf{n}, \mathbf{m}, \overline{\mathbf{m}})$. Then
in terms of this basis the condition~\eqref{dep} that the metric can be
block diagonalised is given by
\begin{align*}
\boldsymbol{\ell} \wedge \mathbf{n} \wedge d\boldsymbol{\ell} &= 0, \\
\boldsymbol{\ell} \wedge \mathbf{n} \wedge d\mathbf{n} &= 0, \\
\mathbf{m} \wedge \overline{\mathbf{m}} \wedge d\mathbf{m} &= 0, \\
\mathbf{m} \wedge \overline{\mathbf{m}} \wedge d\overline{\mathbf{m}} &= 0.
\label{nullform}
\end{align*}
From equation~(4.13.44) in~\cite{PR1}, the above conditions result in reality
constraints on the spin coefficients given by
\begin{equation}
\rho = \overline{\rho}, \qquad \rho' = \overline{\rho'}, \qquad \overline{\tau'} = \tau, \qquad \tau' = \overline{\tau}.
\label{reality}
\end{equation}
We now make use of the Newman--Penrose equations (4.11.12) in~\cite{PR1} to
obtain the equation
\begin{align*}
D'\rho - \delta' \tau + D\rho' - \delta \tau' &= 2 \rho \rho' - \left( \tau \overline{\tau} + \tau' \overline{\tau'} \right) + \rho(\gamma+\overline{\gamma}) + \rho'(\gamma'+\overline{\gamma'})- \left( \tau(\alpha+\overline{\alpha'}) + \overline{\tau}(\alpha'+\overline{\alpha}) \right)
\\
& \hskip 2cm -4\Lambda - 2 \left( \Psi_2 + \kappa \kappa' - \sigma \sigma' \right).
\end{align*}
Because of the reality conditions on the spin coefficients~\eqref{reality}, we
see that the imaginary part of the left hand side of this equation must
vanish. Similarly all the terms but the final one on the right hand side are
real and have vanishing imaginary part. It must therefore be the case
that the final term also has vanishing imaginary part so that
\begin{equation}
\mathrm{Im} \left( \Psi_2 + \kappa \kappa' - \sigma \sigma' \right) = 0.
\label{NPconstraint}
\end{equation}
Therefore, our block-diagonalisation condition necessarily implies that
this constraint must be satisfied. Note that both \eqref{reality} and
\eqref{NPconstraint} are invariant under spin and boost transformations
which reflects the fact that  the $2$-forms $\boldsymbol{\ell} \wedge
\mathbf{n}$ and $\mathbf{m} \wedge \overline{\mathbf{m}}$ are invariant
under such transformations.

To relate this condition to equation~\eqref{curvcond} above we introduce the
complex curvature of the surface spanned by $\mathbf{m} \wedge
\overline{\mathbf{m}}$ which is given by the formula
\[
K = \sigma \sigma' - \Psi_2 - \rho \rho' + \Phi_{11} + \Lambda.
\]
Twice the real part of this gives the Gaussian curvature while twice
the imaginary part
gives the curvature of the connection of the normal bundle, which in view
of the reality conditions on the spin coefficients is given by
\[
\mathrm{Im}\, K=\mathrm{Im}\, \left(\sigma\sigma' -\Psi_2\right)
\]
The corresponding curvature of the connection of the normal bundle to
$\boldsymbol{\ell} \wedge \mathbf{n}$ is obtained by applying the Sachs
$*$-operation (which has
the effect of swapping $\mathbf{m} \wedge \overline{\mathbf{m}}$ with
$\boldsymbol{\ell} \wedge \mathbf{n}$). Under this operation we have
\[
\sigma^*=-\kappa, \qquad \sigma'^*=\kappa', \qquad \Psi_2^*=\Psi_2,
\]
so that the normal curvature is this time given by
\[
\mathrm{Im}\, K^*=\mathrm{Im}\, \left(-\kappa\kappa' -\Psi_2\right)
\]
Finally we note that the full curvature for the orthonormal frame
corresponding to the Newman--Penrose null tetrad is given by
$R_{TXYZ}=-2\, \mathrm{Im}\, \Psi_2$. Hence condition~\eqref{curvcond} becomes
\[
\mathrm{Im}\, K+\mathrm{Im}\, K^*=\mathrm{Im}\, \Psi_2.
\]
Substituting for $\mathrm{Im}\, K$ and $\mathrm{Im}\, K^*$ we again obtain
equation~\eqref{NPconstraint}. Therefore, the constraint obtained from the
Newman--Penrose equations agrees with that obtained from the prolongation
process.

\vs
Finally, with reference to Remark~\ref{rem:cfl}, it should be noted that
the constraints~\eqref{curvcondition} and~\eqref{NPconstraint} that have
arisen via the prolongation procedure are both preserved under
conformal transformations of the metric, $\mathbf{g}$.
This is, again, a manifestation of the fact that our problem
is actually a problem in conformal, rather than Riemannian/Lorentzian, geometry.

\section{Doubly biorthogonal coordinates}
\label{sec:double}

The problem of diagonalising a metric in $3$-dimensions is equivalent to
that of finding three families of $2$-surfaces
\[
f^i(x^1, x^2, x^3) = c^i, \qquad i = 1, 2, 3
\]
that are mutually orthogonal. Given such \lq triply orthogonal\rq\ surfaces
the change of coordinates
\[
x^{i'} = f^i(x^1, x^2, x^3)
\]
brings the metric to diagonal form. Darboux~\cite{Darboux} (see also
Eisenhart~\cite{Eisenhart})) was able to find all triply orthogonal systems
for the flat metric by first giving a condition on two families of
$2$-surfaces that guaranteed the existence of a third family orthogonal to
both.

Let
\begin{align*}
f(x, y, z) &= a = \hbox{constant} ,
\\
g(x, y, z) &= b = \hbox{constant}
\end{align*}
be two $1$-parameter families of $2$-surfaces $S^1_a$ and $S^2_b$. The
normal $1$-form to $S^1_a$ is $df$ and the normal $1$-form to $S^2_b$ is
$dg$. We require these to be orthogonal so that
\begin{equation}
\mathbf{g}(df, dg) = 0.
\label{4.3}
\end{equation}
We now construct a $1$-form {\bom} orthogonal to both $S^1_a$ and $S^2_b$
\begin{equation}
{\bom} = \star \left( df \wedge dg \right).
\label{4.4}
\end{equation}
In order for there to be a $2$-surface mutually orthogonal to both $S^1_a$
and $S^2_b$ we require {\bom} to be surface forming and hence
\begin{equation}
d{\bom} \wedge {\bom} = 0.
\label{4.5}
\end{equation}
Substituting for~\eqref{4.4} into~\eqref{4.5} gives the condition
\begin{equation}
d \left( \star \left( df \wedge dg \right) \right) \wedge \left( df \wedge dg \right) = 0.
\label{4.6}
\end{equation}
When written out in components~\eqref{4.6} takes the form
\[
\epsilon^{cab} \epsilon_{cde} \{ (\nabla_b \nabla^d f) (\nabla^e g) +
(\nabla^d f)(\nabla_b \nabla^e g) \} \epsilon_{akl} (\nabla^k f) (\nabla^l g) = 0,
\]
which can be simplified to read
\begin{equation}
\epsilon^{abc} \nabla_b f \nabla_c g
\left[ (\nabla^d g)(\nabla_d \nabla_a f) - (\nabla^d f)(\nabla_d \nabla_a g) \right] = 0.
\label{4.7}
\end{equation}
On the other hand differentiating~\eqref{4.3} gives
\begin{equation}
(\nabla_b \nabla^a f)(\nabla_b g) + (\nabla_a f)(\nabla_b \nabla^a g) = 0.
\label{4.8}
\end{equation}
We can now use~\eqref{4.8} to replace the second derivatives of $g$
in~\eqref{4.7} by second derivatives of $f$ to obtain:
\[
\epsilon^{abc} (\nabla_b f) (\nabla_c g) (\nabla^d g) (\nabla_d \nabla_a f) = 0.
\]
Now since $\nabla^dg$ is normal to $S^2_b$, it is tangent to $S^1_a$. Hence
if we are given some function $f$ that defines a family of surfaces
$S^1_a$, any surface $S^2_b$ that intersects it orthogonally with the
mutually orthogonal direction surface forming, must intersect $S^1_a$ in a
line with tangent direction $X^a$ that satisfies
\begin{equation}
\epsilon^{abc} (\nabla_b f) X_c X^d (\nabla_d \nabla_a f) = 0.
\label{4.10}
\end{equation}
This is just the classical result that the surfaces intersect in lines of
curvature~\cite{Darboux, Eisenhart}.

The significant point about this is that given $f$ we can
solve~\eqref{4.10} to give $X^a$ algebraically in terms of first and second
derivatives of $f$. Since $X^a$ is tangent to both $S^1_a$ and $S^2_b$ it
is normal to the third surface and must satisfy the surface orthogonal
condition
\[
\epsilon^{abc} (\nabla_a X_b) X_c = 0.
\]
Substituting for $X^a$ we obtain a third-order partial differential
equation for $f$; the Darboux equation~\cite{Darboux}, see also
Eisenhart~\cite{Eisenhart} for details.

We see from the above that the coordinate surface of a triply orthogonal
system must satisfy Darboux's equation. Conversely, given a solution $f(x,
y, z)$ of the Darboux equation one can calculate the lines of curvature of
the surfaces $S^1_a$ given by $f(x, y, z)=a$, and then find an orthogonal
family of surfaces $S^2_b$ which intersects $S^1_a$ orthogonally along
these lines. One then knows that the direction orthogonal to both normals
is surface orthogonal and hence one has a triply orthogonal system of
surfaces. (Note in practice it is often simpler to perform the last two
steps in the opposite order.) Hence all triply orthogonal surface are
determined by solutions to the third-order Darboux partial differential
equation.

In the case of \lq doubly biorthogonal\rq\ coordinate systems we proceed in
a similar manner. We first ask when there exists a family of two surfaces
orthogonal to a given two-parameter family of $2$-surfaces.

Let the given two-parameter family of two surfaces $S_{a, b}$ be given by
\[
f(x, y, z, w) = a,
\qquad
g(x, y, z, w) = b.
\]
Since $df$ and $dg$ are both co-normals to $S$ we require ${\bom} = \star
\left( df \wedge dg \right)$ to be surface-orthogonal. By the Frobenius
theorem this is the condition
\[
\left( \star d{\bom} \right) \wedge \star {\bom} = 0,
\]
which, in components, takes the form
\begin{equation}
\epsilon^{ijkl} (\nabla_j f) (\nabla_k g) \left\{ (\nabla_m f)(\nabla^m \nabla_l g) - (\nabla_m g) (\nabla^m \nabla_l f) \right\} = 0.
\label{4.14}
\end{equation}
If one contracts~\eqref{4.14} with $\nabla_i f$ or $\nabla_i g$ then the
expression vanishes whatever the value of the final term. On the other hand
if one contracts it with an element $\mu_i$ that is not in the linear span
of $\nabla_i f$ and $\nabla_i g$ then $Y^i = \epsilon^{ijkl} \mu_i \nabla_j
f \nabla_k g$ is a non-zero vector orthogonal to $\nabla_i f$ and $\nabla_i
g$. Furthermore any vector $Y^i$ orthogonal to $\nabla_i f$ and $\nabla_i
g$ can be obtained in this way by choosing $\mu_i$ suitably. Hence we
require
\begin{equation}
Y^i \left\{ (\nabla_j f)(\nabla^j \nabla_i g) - (\nabla_j g)(\nabla^j \nabla_i f) \right\} = 0
\mbox{ for all $Y^i$ such that } Y^i \nabla_i f = Y^i \nabla_i g = 0.
\label{4.15}
\end{equation}
This gives a pair of coupled second-order equations for $f$ and $g$. Note
that, unlike the case of triply orthogonal systems, $g^{ij} \nabla_i f
\nabla_j g \neq 0$ in general since we cannot be expected to diagonalise
one of the $2 \times 2$ blocks as well as obtain block diagonal form (this
would involve setting five terms in the metric to zero). Hence there is no
possibility of eliminating the second derivative of $g$ in favour of
derivatives of $f$ as was done in three dimensions. Indeed~\eqref{4.15}
implies~\eqref{4.14} and hence that ${\bom} = \star \left( df \wedge dg
\right)$ is surface orthogonal. Thus~\eqref{4.15} is a necessary and
sufficient condition for the existence of a doubly biorthogonal coordinate
system.

\begin{proposition}
All doubly biorthogonal systems are determined by solutions to the pair of
coupled second-order partial differential equations
\[
Y^i \left\{ (\nabla_j f) (\nabla^j \nabla_i g) - (\nabla_j g) (\nabla^j \nabla_i f) \right\} = 0
\mbox{ for all $Y^i$ such that } Y^i \nabla_i f = Y^i \nabla_i g = 0.
\]
\end{proposition}

\appendix
\section{Results from the theory of exterior differential systems}
\label{app:XDS}

We now recall some standard definitions and results from the theory of
exterior differential systems. For more information, see~\cite{BCGGG}, the
terminology and notation of which we will generally follow.

\vs Throughout this section, let $M$ be an arbitrary smooth manifold of
dimension $n$. Let $\Omega^p(M)$ denote the space of $C^{\infty}$ sections
of $\bigwedge^p T^* M$ and $\Omega^*(M) := \bigoplus_{p=0}^n \Omega^p(M)$.

\vs

An \emph{\XDS\/}, $\mathcal{I}$, on $M$ consists of a two-sided,
homogeneous differential ideal, $\mathcal{I} \subset \Omega^*(M)$. In
particular, we have
\begin{itemize}
\item[$\bullet$] Given $\boldsymbol{\alpha} \in \mathcal{I}$, then
$\boldsymbol{\alpha} \wedge \boldsymbol{\beta} \in \mathcal{I}$ and
$\boldsymbol{\beta} \wedge \boldsymbol{\alpha} \in \mathcal{I}$ for all
$\boldsymbol{\beta} \in \Omega^*(M)$.
\item[$\bullet$] $\mathcal{I} = \bigoplus \mathcal{I}^q$ where
$\mathcal{I}^q := \mathcal{I} \cap \Omega^q(M)$ and, for any
$\boldsymbol{\alpha} \in \mathcal{I}$, the part of $\boldsymbol{\alpha} \in
\mathcal{I}$ lying in $\mathcal{I}^q$ also lies in $\mathcal{I}$, for $q =
0, \dots, n$.
\item[$\bullet$] For all $\boldsymbol{\alpha} \in \mathcal{I}$ we have
$d\boldsymbol{\alpha} \in \mathcal{I}$.
\end{itemize}
Given a point $x \in M$, a $k$-dimensional linear subspace $E_k \subseteq
T_x M$ (where $k \in \{ 1, \dots, n \}$) is an \emph{integral element of
$(\mathcal{I}, \bOm)$ (of dimension $k$) based at $x$\/} if
$\restrict{\bvar}{E_k} = 0$ for all ${\bvar} \in \mathcal{I}$, where
$\restrict{\boldsymbol{\alpha}}{E_k}$ denotes the restriction of a form
$\boldsymbol{\alpha}$ to $E_k$. The set of integral elements of
$\mathcal{I}$ of dimension $k$ is denoted $V_k(\mathcal{I})$.

An \emph{\XDSIC\/}, $(\mathcal{I}, \bOm)$, on $M$ consists of an {\XDS}
$\mathcal{I} \subset \Omega^*(M)$, and a non-vanishing differential form
$\bOm \in \Omega^p(M)$. Given a point $x \in M$, an $p$-dimensional linear
subspace $E_p \subseteq T_x M$ is an \emph{integral element of
$(\mathcal{I}, \bOm)$ based at $x$\/} if $\restrict{\bvar}{E_p} = 0$ for
all ${\bvar} \in \mathcal{I}$ and $\restrict{\bOm}{E_p} \neq 0$. The set of
integral elements of $(\mathcal{I}, \bOm)$ is denoted $V_p(\mathcal{I},
\bOm)$.

\begin{definition}
An \emph{integral manifold\/} of $(\mathcal{I}, \bOm)$ is an immersed
sub-manifold $i \colon N \rightarrow M$ with the property that $i^* \bvar = 0$,
for all ${\bvar} \in \mathcal{I}$, and $i^* {\bOm} \neq 0$. Equivalently,
$i_* \left( T_x N \right) \subset T_{i(x)} M$ should be an integral element
of $(\mathcal{I}, \bOm)$, for each $x \in N$.
\end{definition}

\begin{definition}
An \emph{integral flag of $(\mathcal{I}, \bOm)$ based at $x$} is a nested
sequence of subspaces $(0)_x \subset E_1 \subset E_2 \subset \dots \subset
E_p \subseteq T_x M$, with the properties that
\begin{itemize}
\item[$\bullet$] $E_k$ is of dimension $k$, for $k = 0, \dots, p-1$;
\item[$\bullet$] $E_p$ is an integral element of $(\mathcal{I}, \bOm)$.
\end{itemize}
\end{definition}

\begin{definition}
Let $\mathbf{e}_1, \dots, \mathbf{e}_k$ be a basis for $E_k \subseteq T_x
M$. The \emph{polar space\/} of $E$ is the vector space
\[
H(E) = \left\{\mathbf{v} \in T_x M: \vphantom{|^|} {\bvar}(\mathbf{v}, \mathbf{e}_1, \dots, \mathbf{e}_k) = 0 \mbox{ for all ${\bvar} \in \left. \mathcal{I}^{k+1} \right|_x$} \right\}.
\]
\label{defn:polarspace}
\end{definition}

\begin{definition}
Let $(0)_x \subset E_1 \subset E_2 \subset \dots \subset E_p \subseteq T_x
M$ be an integral flag of $(\mathcal{I}, \bOm)$ based at $x \in M$. We
define the integers $\{ c_k: k = -1, 0, \dots, p \}$ as follows:
\[
c_k = \begin{cases} 0 &k=-1, \\ \codim H(E_k) &k=0, \dots, p-1 \\ \dim M - p &k=p. \end{cases}
\]
\label{defn:A.3}
\end{definition}

\vs
We now quote the first half of Theorem~1.11 from Chapter~III of \cite{BCGGG}:

\begin{proposition}
Let $(\mathcal{I}, \bOm)$ be an {\XDSIC} on manifold $M$, where
$\mathcal{I}$ contains no non-zero forms of degree $0$. Let $(0)_x \subset
E_1 \subset E_2 \subset \dots \subset E_p \subset T_x M$ be an integral
flag of $(\mathcal{I}, \bOm)$. Then $V_p(\mathcal{I}, \bOm) \subseteq
Gr_p(TM)$ is of codimension at least $c_0 + c_1 + \dots + c_{p-1}$ at
$E_p$.
\end{proposition}

If there exists a neighbourhood, $U$ of $E_p$ in $Gr_p(TM)$ such that
$V_p(\mathcal{I}, \bOm) \cap U$ is a smooth sub-manifold of codimension
$c_0 + c_1 + \dots + c_{p-1}$ in $U$ at $E_p$, then we say that the
integral flag $E_p$ \emph{passes Cartan's test}.

\vs
The key result is the following:

\begin{theorem}[Cartan--K\"{a}hler Theorem: \cite{BCGGG}, Chapter~III, Corollary~2.3]
\label{CKT}
Let $(\mathcal{I}, \bOm)$ be an analytic differential ideal on a manifold
$M$. Let $E_p \subset T_x M$ be an integral element of $(\mathcal{I},
\bOm)$ that passes Cartan's test. Then there exists an integral manifold of
$(\mathcal{I}, \bOm)$ through $x$, the tangent space to which, at $x$, is
$E_p$.
\end{theorem}

\subsection{Linear Pfaffian systems}
A \emph{Pfaffian system} is an {\XDSIC}, $(\mathcal{I}, \bOm)$, on a
manifold $M$ such that $\mathcal{I}$ is generated, as a differential ideal,
by sections of a sub-bundle $I \subset T^* M$. (It is assumed that $I$ is
of constant rank, $s_0$.) The independence condition, $\bOm$, may be
characterised by a sub-bundle $J \subset T^* M$, with $I \subset J \subset
T^* M$ and $\rank J/I = n$, in which case $\bOm$ corresponds to a
non-vanishing section of $\wedge^n \left( J/I \right)$. Such a Pfaffian
system is \emph{linear\/} if
\[
dI \equiv 0 \mod{J}.
\]

Locally, we may choose a coframe $\{ {\bth}^1, \dots, {\bth}^{s_0},
{\bom}^1, \dots, {\bom}^n, {\bpi}^1, \dots, {\bpi}^t \}$ on $M$ such that
$I_x = \spn({\bth}^1, \dots, {\bth}^{s_0})$, $J_x = \spn({\bth}^1, \dots,
{\bth}^{s_0}, {\bom}^1, \dots, {\bom}^n)$. In this case, the condition that
the Pfaffian system be linear is that there exist functions
$A^a{}_{\varepsilon i}$, $c^a{}_{ij}$ on $M$ such that
\begin{equation}
d{\bth}^a \equiv \sum_{\varepsilon, i} A^a{}_{\varepsilon i} {\bpi}^{\varepsilon} \wedge {\bom}^i + \frac{1}{2} \sum_{i, j} c^a{}_{ij} {\bom}^i \wedge {\bom}^j \mod{\bth}.
\label{lPs}
\end{equation}
Under a change of coframe of the form
\begin{equation}
({\bth}^{\sigma}, {\bom}^i, {\bpi}^{\varepsilon}) \mapsto ({\bth}^{\sigma}, {\bom}^i, {\bpi}^{\varepsilon} + \sum_i p^{\varepsilon}{}_i {\bom}^i),
\label{changeofcoframe}
\end{equation}
the coefficients $c^a{}_{ij}$ transform according to the rule
\[
c^a{}_{ij} \mapsto c^a{}_{ij} + \sum_{\varepsilon} \left( A^a{}_{\varepsilon i} \, p^{\varepsilon}{}_j - A^a{}_{\varepsilon j} \, p^{\varepsilon}{}_i \right).
\]
We define two collections of coefficients $c^a{}_{ij}$,
$\widetilde{c}^a{}_{ij}$ to be equivalent if there exists parameters
$p^{\varepsilon}{}_i$ such that $\widetilde{c}^a{}_{ij} = c^a{}_{ij} +
\sum_{\varepsilon} \left( A^a{}_{\varepsilon i} \, p^{\varepsilon}{}_j -
A^a{}_{\varepsilon j} \, p^{\varepsilon}{}_i \right)$, and denote the
corresponding equivalence class of coefficients by $\left[ c
\right]$. $\left[ c \right]$ is the \emph{essential torsion\/} of the
linear Pfaffian system $(\mathcal{I}, \bOm)$. If it is possible to choose
the $p^{\varepsilon}{}_i$ such that $\widetilde{c}^a{}_{ij} = 0$
(i.e. there is no essential torsion) then we say that the torsion can be
\emph{absorbed}. Given a point $x \in M$, there exists an integral element
of $(\mathcal{I}, \bOm)$ based at $x$ if and only if $\left[ c \right](x) =
0$.

In the terminology of Olver~\cite[pp.~351]{Olver}, the \emph{degree of
indeterminacy\/}, $r^{(1)}$, of the above coframe is the number of the
number of solutions of the homogeneous problem
\[
\sum_{\varepsilon} \left( A^a{}_{\varepsilon i} \, p^{\varepsilon}{}_j - A^a{}_{\varepsilon j} \, p^{\varepsilon}{}_i \right) = 0.
\]
Equivalently, it is the number of transformations of the form~\eqref{changeofcoframe} that leave the structure equations~\eqref{lPs} unchanged.

If the torsion vanishes on an open neighbourhood, $U$, of $x$, then we write~\eqref{lPs} in the form
\begin{equation}
d{\bth}^a \equiv \sum_i {\bpi}^a{}_i \wedge {\bom}^i \mod{\bth},
\label{lPs2}
\end{equation}
where ${\bpi}^a{}_i \equiv \sum_{\varepsilon, i} A^a{}_{\varepsilon i} {\bpi}^{\varepsilon} \mod{\{ {\bth}, {\bom} \}}$.

\vs To determine the involutivity of a torsion-free linear Pfaffian system
at $x \in M$, we need to consider its \emph{tableau} $A_x$, which is a
linear subspace of $I_x^* \otimes \left( J_x / I_x \right)$. For our
purposes, however, it is simpler (but equivalent) to consider the
corresponding tableau matrix:
\begin{definition}
Given a linear Pfaffian system with structure equations as in~\eqref{lPs2}
and a point $x \in M$, the \emph{tableau matrix at $x$\/} is the $s_0
\times n$ matrix of elements of $T_x^* M / J_x$ given by
\[
{\pi}_x = \left( {\bpi}^a{}_i(x) \right) \mod{\{ {\bth}(x), {\bom}(x) \}}.
\]
The reduced Cartan characters, $s_1^{\prime}, \dots, s_4^{\prime}$, of the
tableau $A_x$ are defined by
\[
s_1^{\prime} + \dots + s_k^{\prime} = \mbox{ the number of linearly-independent $1$-forms in the first $k$ columns of ${\pi}_x$},
\]
for a generic choice of the $1$-forms $\{ {\bom}^i \}$.
\end{definition}

In order to check for involutivity of the system $(\mathcal{I}, \bOm)$ at
$x \in M$, we need to know the dimension of the first prolongation,
$A^{(1)}$, of the tableau $A_x$. We do not give a formal definition of
$A^{(1)}$, but content ourselves with the following characterisation, which
gives us sufficient information to calculate its dimension:
\begin{proposition}[\cite{IL}, Proposition 5.7.1]
\label{5.7.1}
Let $x \in M$ and ${\bpi}^a{}_i \in T^*_x M$ satisfy $d{\bth}^a \equiv
{\bpi}^a{}_i \wedge {\bom}^i \mod{\bth}$. Then the first prolongation,
$A^{(1)}$, of the tableau $A_x$ may be identified with the space of
$1$-forms ${\widetilde{\bpi}}{}^a{}_i \equiv {\bpi}^a{}_i \mod{\bth}$ such
that $d{\bth}^a \equiv {\widetilde{\bpi}}^a{}_i \wedge {\bom}^i
\mod{\bth}$.
\end{proposition}

\begin{remark}
Proposition~\ref{5.7.1} implies that $\dim A^{(1)}$ is equal to the degree
of indeterminacy, $r^{(1)}$ of the coframe. Therefore, in this notation, a
Pfaffian system is involutive if it satisfies
\[
s_1^{\prime} + 2 s_2^{\prime} + \dots + n s_n^{\prime} = r^{(1)}.
\]
\end{remark}

\begin{proposition}[\cite{BCGGG}, pp.~318]
The first prolongation of the tableau, $A_x$, and the reduced Cartan
characters obey the inequality
\[
\dim A^{(1)} \le s_1^{\prime} + 2 s_2^{\prime} + \dots + n s_n^{\prime}.
\]
The tableau, $A_x$, is \emph{involutive\/} if equality holds in this equation.
\end{proposition}

\begin{proposition}[\cite{BCGGG}, Chapter~IV, Theorem~5.16]
\label{5.16}
The linear Pfaffian system $(\mathcal{I}, \bOm)$ is involutive at $x \in M$
if and only if
\begin{itemize}
\item[(i)] $[c](x) = 0$;
\item[(ii)] the tableau $A$ is involutive.
\end{itemize}
\end{proposition}

\begin{remark}
\label{PfaffianCK}
If the system $(\mathcal{I}, \bOm)$ is involutive at $x \in M$, then the
Cartan--K\"{a}hler theorem implies the existence of an integral manifold of
the system $(\mathcal{I}, \bOm)$ through the point $x$.
\end{remark}

\section{Absorption formulae}

\subsection{Explicit absorption procedures}
\label{sec:absorb}
The structure equations for the Pfaffian system $(\mathcal{I}, {\bOm})$ on
the manifold $M^{(1)}$ are given in equation~\eqref{dtheta}. We can absorb
most of the torsion in the original problem by setting
\begin{align*}
\boldpi{1}{1}{2} &= d\llambda{1}{1}{2} + T^1{}_{212} \, {\bom}^2 + T^1{}_{213} \, {\bom}^3 + T^1{}_{214} \, {\bom}^4,
\\
\boldpi{2}{1}{2} &= d\llambda{2}{1}{2} + T^1{}_{223} \, {\bom}^3 + T^1{}_{224} \, {\bom}^4,
\\
\boldpi{3}{1}{2} &= d\llambda{3}{1}{2} + T^1{}_{234} \, {\bom}^4,
\\
\boldpi{4}{1}{2} &= d\llambda{4}{1}{2},
\\
\boldpi{1}{1}{3} &= d\llambda{1}{1}{3} + T^1{}_{312} \, {\bom}^2 + T^1{}_{313} \, {\bom}^3 + T^1{}_{314} \, {\bom}^4,
\\
\boldpi{2}{1}{3} &= \boldpi{1}{2}{3} = d\llambda{2}{1}{3} + T^1{}_{324} \, {\bom}^4,
\\
\boldpi{3}{1}{3} &= d\llambda{3}{1}{3} - T^1{}_{323} \, {\bom}^2 + T^1{}_{334} \, {\bom}^4,
\\
\boldpi{4}{1}{3} &= {\bpi}_3{}^1{}_4 = d\llambda{4}{1}{3},
\\
\boldpi{1}{1}{4} &= d\llambda{1}{1}{4} + T^1{}_{412} \, {\bom}^2 + T^1{}_{413} \, {\bom}^3 + T^1{}_{414} \, {\bom}^4,
\\
\boldpi{2}{1}{4} &= \boldpi{1}{2}{4} = d\llambda{2}{1}{4} + T^1{}_{423} \, {\bom}^3,
\\
\boldpi{4}{1}{4} &= d\llambda{4}{1}{4} - T^1{}_{424} \, {\bom}^2 - T^1{}_{434} \, {\bom}^3,
\\
\boldpi{2}{2}{3} &= d\llambda{2}{2}{3} - T^2{}_{312} \, {\bom}^1 + T^2{}_{323} \, {\bom}^3 + T^2{}_{324} \, {\bom}^4,
\\
\boldpi{3}{2}{3} &= d\llambda{3}{2}{3} - T^2{}_{313} \, {\bom}^1 + T^2{}_{334} \, {\bom}^4,
\\
\boldpi{4}{2}{3} &= \boldpi{3}{2}{4} = d\llambda{4}{2}{3} + \left( T^1{}_{324} + T^2{}_{341} \right) \, {\bom}^1,
\\
\boldpi{2}{2}{4} &= d\llambda{2}{2}{4} - T^2{}_{412} \, {\bom}^1 + T^2{}_{423} \, {\bom}^3 + T^2{}_{424} \, {\bom}^4,
\\
\boldpi{4}{2}{4} &= d\llambda{4}{2}{4} - T^2{}_{414} \, {\bom}^1 - T^2{}_{434} \, {\bom}^3,
\\
\boldpi{1}{3}{4} &= d\llambda{1}{3}{4} + T^3{}_{412} \, {\bom}^2 + T^3{}_{413} \, {\bom}^3 + T^3{}_{414} \, {\bom}^4,
\\
\boldpi{2}{3}{4} &= d\llambda{2}{3}{4} + T^3{}_{423} \, {\bom}^3 + T^3{}_{424} \, {\bom}^4,
\\
\boldpi{3}{3}{4} &= d\llambda{3}{3}{4} + T^3{}_{434} \, {\bom}^4,
\\
\boldpi{4}{3}{4} &= d\llambda{4}{3}{4}.
\end{align*}

The structure equations then take the form given in
equation~\eqref{dbth}. Note that the quantity on the left-hand-side of
equation~\eqref{Tnot0} is invariant under transformations of the form
$\boldpi{a}{b}{c} \rightarrow \boldpi{a}{b}{c} + \delta \boldpi{a}{b}{c}$
with $\delta \boldpi{a}{b}{c} = \sum_d \Pi_a{}^b{}_{cd} \bom^d$ that
preserve the required symmetries of the $\boldpi{a}{b}{c}$
(i.e. $\boldpi{1}{2}{3} = \boldpi{2}{1}{3}$). As such, it follows that, at
points of $M^{(1)}$ at which $T(x, g, \lambda) \neq 0$, there remains
essential torsion in the system that cannot be absorbed into a redefinition
of the $1$-forms $\boldpi{a}{b}{c}$.

\subsection{Calculation of degree of indeterminacy}
\label{sec:absorb2}
We let $\mathbf{X} := \left( y^1, \dots, y^8 \right) \in \mathbb{R}^{4, 4}$
with the split-signature metric
\[
\mathbf{q}(\mathbf{X}, \mathbf{X}) := 2 \left( y^1 y^2 - y^3 y^4 + y^5 y^6 - y^7 y^8 \right).
\]
Then our constraint equation~\eqref{constraint1} takes the
\begin{equation}
T(x, g, \mathbf{X}) := \mathbf{q}(\mathbf{X}, \mathbf{X}) + R_{1234}(x, g) = 0.
\label{constraint1b}
\end{equation}
We then need to consider the pull-back to $S$ of the exterior derivative of $T$, and find that
\begin{equation}
i^* (dT) = \widetilde{y}^1 \widetilde{\bpi}{}^2 + \widetilde{y}^2 \widetilde{\bpi}{}^1 - \widetilde{y}^3 \widetilde{\bpi}{}^4 - \widetilde{y}^4 \widetilde{\bpi}{}^3 + \widetilde{y}^5 \widetilde{\bpi}{}^6 + \widetilde{y}^6 \widetilde{\bpi}{}^5 - \widetilde{y}^7 \widetilde{\bpi}{}^8 - \widetilde{y}^8 \widetilde{\bpi}{}^7 + \sum_a \Psi_a \widetilde{\bom}^a \equiv 0 \mod{ {\widetilde{\bth}}}.
\label{lindep3}
\end{equation}
Note that the $1$-forms $\{ {\widetilde{\bpi}}^{\alpha},
{\widetilde{\bro}}^a, {\widetilde{\bmu}}^a, {\widetilde{\bnu}}^a \}$ are
not uniquely determined by the structure
equations~\eqref{structureequations} and~\eqref{tableau2}. In particular,
we are free to consider variations of the form
\begin{subequations}
\begin{align}
\widetilde{\bpi}{}^{\alpha} &\mapsto \widetilde{\bpi}{}^{\alpha} + \delta \widetilde{\bpi}{}^{\alpha},
&\widetilde{\bro}{}^i &\mapsto \widetilde{\bro}{}^i + \delta \widetilde{\bro}{}^i,
\label{pirho}
\\
{\widetilde{\bmu}}^a &\mapsto {\widetilde{\bmu}}^a + \delta{\widetilde{\bmu}}^a,
&{\widetilde{\bnu}}^a &\mapsto {\widetilde{\bnu}}^a + \delta{\widetilde{\bnu}}^a
\label{pirho2}\end{align}\end{subequations}
with
\begin{equation}
\delta\widetilde{\bpi}{}^{\alpha}, \delta\widetilde{\bro}{}^a, \delta{\widetilde{\bmu}}^a, \delta{\widetilde{\bnu}}^a \equiv 0 \mod{\widetilde{\bom}{}^a},
\label{pirho3}
\end{equation}
as long as they preserve~\eqref{structureequations}
and~\eqref{tableau2}. We first wish to show that, in the generic case where
$\widetilde{y}^1, \dots, \widetilde{y}^8$ are all non-zero, we may use such
transformations to absorb the $\sum_a \Psi_a \widetilde{\bom}^a$ term
in~\eqref{lindep3} into a redefinition of the $1$-forms
$\widetilde{\bpi}{}^{\alpha}$, $\widetilde{\bro}{}^i$.

Firstly, it is straightforward to show that the most general variation that
preserves the structure equations~\eqref{structureequations}
and~\eqref{tableau} is of the form (from now on, we drop tildes on all
quantities)
\begin{align*}
\delta {\bpi}{}^1 &= \alpha {\bom}^1 + \beta {\bom}^2 + \gamma {\bom}^3 + \delta {\bom}^4,
\\
\delta {\bpi}{}^3 &= \epsilon {\bom}^1 + \zeta {\bom}^2 + \delta {\bom}^3 + \eta {\bom}^4,
\\
\delta {\bpi}{}^5 &= \theta {\bom}^1 + \delta {\bom}^2 + \iota {\bom}^3 + \kappa {\bom}^4,
\\
\delta {\bpi}{}^7 &= \delta {\bom}^1 + \lambda {\bom}^2 + \mu {\bom}^3 + \nu {\bom}^4,
\end{align*}
along with
\begin{align*}
\delta {\bpi}{}^2 &= \xi {\bom}^1 + o {\bom}^2 + \frac{1}{2} \left( \lambda - \theta \right) {\bom}^3 + \frac{1}{2} \left( \pi - \rho \right) {\bom}^4,
\\
\delta {\bpi}{}^4 &= \sigma {\bom}^1 + \tau {\bom}^2 + \frac{1}{2} \left( \upsilon - \phi \right) {\bom}^3 + \frac{1}{2} \left( \lambda - \theta \right) {\bom}^4,
\\
\delta {\bpi}{}^6 &= \frac{1}{2} \left( \gamma - \eta \right) {\bom}^1 + \frac{1}{2} \left( \upsilon - \pi \right) {\bom}^2 + \chi {\bom}^3 + \psi {\bom}^4,
\\
\delta {\bpi}{}^8 &= \frac{1}{2} \left( \phi - \rho \right) {\bom}^1 + \frac{1}{2} \left( \gamma - \eta \right) {\bom}^2 + \omega{\bom}^3 + \Omega {\bom}^4,
\end{align*}
and
\begin{align*}
\delta {\bro}{}^1 &= (\zeta - \xi) {\bom}^1 + (o + \epsilon) {\bom}^2 + \frac{1}{2} \left( \lambda + \theta \right) {\bom}^3 + \frac{1}{2} \left( \pi + \rho \right) {\bom}^4,
\\
\delta {\bro}{}^2 &= (\beta - \sigma) {\bom}^1 + (\tau + \alpha) + \frac{1}{2} \left( \upsilon + \phi \right) {\bom}^3 + \frac{1}{2} \left( \lambda + \theta \right) {\bom}^4,
\\
\delta {\bro}{}^3 &= - \frac{1}{2} \left( \gamma + \eta \right) {\bom}^1 - \frac{1}{2} \left( \upsilon + \pi \right) {\bom}^2 - (\chi + \nu) {\bom}^3 - (\phi + \mu) {\bom}^4,
\\
\delta {\bro}{}^4 &= - \frac{1}{2} \left( \phi + \rho \right) {\bom}^1 - \frac{1}{2} \left( \gamma + \eta \right) {\bom}^2 - (\omega + \kappa) {\bom}^3 - (\delta - \Omega) {\bom}^4,
\end{align*}
where $\alpha, \dots, \omega$ and $\Omega$ are $25$ free parameters. We now wish to find a transformation of the form~\eqref{pirho} with the property that
\[
y^1 \delta{\bpi}{}^2 + y^2 \delta{\bpi}{}^1 - y^3 \delta{\bpi}{}^4 - y^4 \delta{\bpi}{}^3 + y^5 \delta{\bpi}{}^6 + y^6 \delta{\bpi}{}^5 - y^7 \delta{\bpi}{}^8 - y^8 \delta{\bpi}{}^7 = - \sum_a \Psi_a {\bom}^a.
\]
Using the form of $\delta{\bpi}{}^{\alpha}$ given above, this implies that we need to find vectors $\mathbf{Y}_1, \dots, \mathbf{Y}_4$ of the form
\begin{align*}
\mathbf{Y}_1 &= \left( \alpha, \xi, \epsilon, \sigma, \theta, \frac{1}{2} \left( \gamma - \eta \right), \delta, \frac{1}{2} \left( \phi - \rho \right) \right),
\\
\mathbf{Y}_2 &= \left( \beta, o, \zeta, \tau, \delta, \frac{1}{2} \left( \phi - \rho \right), \lambda, \frac{1}{2} \left( \gamma - \eta \right) \right),
\\
\mathbf{Y}_3 &= \left( \gamma, \frac{1}{2} \left( \lambda - \theta \right), \delta, \frac{1}{2} \left( \upsilon - \phi \right), \iota, \chi, \mu, \omega \right),
\\
\mathbf{Y}_4 &= \left( \delta, \frac{1}{2} \left( \pi - \rho \right), \eta, \frac{1}{2} \left( \lambda - \theta \right), \kappa, \psi, \nu, \Omega \right),
\end{align*}
with the property that
\begin{equation}
\mathbf{q}(\mathbf{X}, \mathbf{Y}_i) = - \Psi_i, \qquad i = 1, \dots, 4.
\label{Yeqns}
\end{equation}
In the generic case where $y^1, \dots, y^8$ are all non-zero, these
equations may be solved for four of the free parameters in the
$\mathbf{Y}_i$, and hence will yield the required
transformation~\eqref{pirho} in terms of the remaining $21$ free
parameters. Substituting these expressions into $\delta{\bpi}^{\alpha}$, we
therefore generate a $21$-parameter family of $1$-forms
${{\bpi}^{\prime}}{}^{\alpha} := {\bpi}{}^{\alpha} + \delta
{\bpi}{}^{\alpha}$, ${{\bro}^{\prime}}{}^i := {\bro}{}^i + \delta
{\bro}{}^i$ in terms of which the constraint equation~\eqref{lindep3} takes
the required form
\begin{equation}
y^1 {{\bpi}^{\prime}}{}^2 + y^2 {{\bpi}^{\prime}}{}^1 - y^3 {{\bpi}^{\prime}}{}^4 - y^4 {{\bpi}^{\prime}}{}^3 + y^5 {{\bpi}^{\prime}}{}^6 + y^6 {{\bpi}^{\prime}}{}^5 - y^7 {{\bpi}^{\prime}}{}^8 - y^8 {{\bpi}^{\prime}}{}^7 \equiv 0 \mod{ {{\bth}}}.
\label{lindep4}
\end{equation}

\vs Finally, based on the preceding calculations, we deduce
Proposition~\ref{prop41}:
\begin{proof}[Proof of Proposition~\ref{prop41}]
Since we are dealing with a linear Pfaffian system, the first prolongation
of $A_p$ is necessarily an affine-linear space (cf.~\cite{BCGGG},
Chapter~IV) the dimension of which, by Proposition~\ref{5.7.1}, is equal to
$r^{(1)}$, the degree of indeterminacy of our coframe. By definition,
$r^{(1)}$ is equal to the number of parameters in a change of the $1$-forms
as in equations~\eqref{pirho}, \eqref{pirho2} and~\eqref{pirho3} that
preserve the form of the structure equations~\eqref{structureequations}
and~\eqref{tableau}. Setting $\Psi_a = 0$ in the calculations above, we see
that there exists a $21$-parameter family of $1$-forms,
$\delta{\widetilde{\bpi}}^\alpha, \delta{\widetilde{\bro}}^a$ on $S$ that
satisfy these conditions. In addition, we have $10$ free parameters in the
choice of $\delta{\widetilde{\bmu}}^a$ and $10$ free parameters in the
choice of $\delta{\widetilde{\bnu}}^a$ consistent with the structure
equations. In total, therefore, at a generic point $p \in S$, we have $41$
free parameters in choosing the $1$-forms in a way that is consistent with
the structure equations.

Therefore $\dim A^{(1)} = r^{(1)} = 41$, as required.
\end{proof}

\newcommand{\etalchar}[1]{$^{#1}$}
\providecommand{\bysame}{\leavevmode\hbox to3em{\hrulefill}\thinspace}
\providecommand{\MR}{\relax\ifhmode\unskip\space\fi MR }
\providecommand{\MRhref}[2]{%
  \href{http://www.ams.org/mathscinet-getitem?mr=#1}{#2}
}
\providecommand{\href}[2]{#2}


\begin{thebibliography}{BCG{\etalchar{+}}91}

\bibitem[BCG{\etalchar{+}}91]{BCGGG}
R.~L. Bryant, S.~S. Chern, R.~B. Gardner, H.~L. Goldschmidt, and P.~A.
  Griffiths, \emph{Exterior differential systems}, Mathematical Sciences
  Research Institute Publications, vol.~18, Springer-Verlag, New York, 1991.
  \MR{MR1083148 (92h:58007)}

\bibitem[Car45]{Cartan}
{\'E}lie Cartan, \emph{Les syst\`emes diff\'erentiels ext\'erieurs et leurs
  applications g\'eom\'etriques}, Actualit\'es Sci. Ind., no. 994, Hermann et
  Cie., Paris, 1945. \MR{MR0016174 (7,520d)}

\bibitem[Dar98]{Darboux}
Gaston Darboux, \emph{Le\c{c}ons sur les syst\`emes orthogonaux et les
  coordonn\'ees curvilignes}, Gauthier-Villars, 1898.

\bibitem[dS80]{dIS}
R.~A. d'Inverno and J.~Smallwood, \emph{Covariant {$2+2$} formulation of the
  initial value problem in general relativity}, Phys. Rev. D (3) \textbf{22}
  (1980), no.~6, 1233--1247. \MR{MR586699 (82c:83011)}

\bibitem[DY84]{DeTY}
Dennis~M. DeTurck and Deane Yang, \emph{Existence of elastic deformations with
  prescribed principal strains and triply orthogonal systems}, Duke Math. J.
  \textbf{51} (1984), no.~2, 243--260. \MR{MR747867 (86b:73014)}

\bibitem[Eis60]{Eisenhart}
Luther~Pfahler Eisenhart, \emph{A treatise on the differential geometry of
  curves and surfaces}, Dover Publications Inc., New York, 1960. \MR{MR0115134
  (22 \#5936)}

\bibitem[Gau22]{Gauss}
C.~F. Gauss, \emph{Allgemeine {A}ufl\"{o}sung der {A}ufgabe die {T}heile einer
  gegebenen {F}l\"{a}che auf einer andern gegebnen {F}l\"{a}che so abzubilden,
  dass die {A}bbildung dem {A}bgebildeten in den kleinsten {T}heilen
  \"{a}hnlich wird}, See Gauss Werke IV, 1822, pp.~189--216.

\bibitem[IL03]{IL}
Thomas~A. Ivey and J.~M. Landsberg, \emph{Cartan for beginners: differential
  geometry via moving frames and exterior differential systems}, Graduate
  Studies in Mathematics, vol.~61, American Mathematical Society, Providence,
  RI, 2003. \MR{MR2003610 (2004g:53002)}

\bibitem[Kin97]{Kini}
D.A. Kini, \emph{{W}eak {S}ingularities in {G}eneral {R}elativity}, Ph.D.
  thesis, University of Southampton (1997).

\bibitem[Olv95]{Olver}
Peter~J. Olver, \emph{Equivalence, invariants, and symmetry}, Cambridge
  University Press, Cambridge, 1995. \MR{MR1337276 (96i:58005)}

\bibitem[PR87]{PR1}
Roger Penrose and Wolfgang Rindler, \emph{Spinors and space-time. {V}ol.\ 1},
  Cambridge Monographs on Mathematical Physics, Cambridge University Press,
  Cambridge, 1987, Two-spinor calculus and relativistic fields. \MR{MR917488
  (88h:83009)}

\bibitem[Spi75]{Spivak}
Michael Spivak, \emph{A comprehensive introduction to differential geometry.
  {V}ol. {IV}}, Publish or Perish Inc., Boston, Mass., 1975. \MR{MR0394452 (52
  \#15254a)}

\bibitem[Tay81]{TaylorPsiD}
Michael~E. Taylor, \emph{Pseudodifferential operators}, Princeton Mathematical
  Series, vol.~34, Princeton University Press, Princeton, N.J., 1981.
  \MR{MR618463 (82i:35172)}

\bibitem[Tod92]{Tod}
K.~P. Tod, \emph{On choosing coordinates to diagonalize the metric}, Classical
  Quantum Gravity \textbf{9} (1992), no.~7, 1693--1705. \MR{MR1173287
  (93e:83026)}

\end{thebibliography}
\end{document}